\RequirePackage{fix-cm}
\RequirePackage{rotating}
\documentclass[smallcondensed]{svjour3}     % onecolumn (ditto)
\smartqed  % flush right qed marks, e.g. at end of proof
\usepackage{graphicx}
\usepackage{color}
\usepackage{amsmath}
\usepackage{algorithm}
\usepackage{comment}
\usepackage{hyperref}
\usepackage{booktabs}       % professional-quality tables
\usepackage{tablefootnote}
\usepackage[utf8]{inputenc}
\usepackage[T1]{fontenc}
\newtheorem{assumption}{Assumption}

\def\sS{\mathcal{S}}
\def\sN{\mathcal{N}}

\def\bR{\mathbb{R}}

\def\sL{\mathcal{L}}
\def\sK{\mathcal{K}}
\def\sF{\mathcal{F}}

\def\sO{\mathcal{O}}
\def\tr{\tilde{r}}
\def\Ktotal{\bar{K}_{\mathrm{total}}}

\def\rev#1{{{\color{black}#1}}}
\def\Rev#1{{{\color{black}#1}}}
\newcommand{\Cmeo}{\mathcal{C}_{\mbox{\rm\scriptsize meo}}}
\newcommand{\Nmeo}{N_{\mbox{\rm\scriptsize meo}}}
\newcommand{\Flow}{F_{\mbox{\rm\scriptsize low}}}
\newcommand{\CNCG}{C_{\mbox{\rm\scriptsize NCG}}}

\def\bZ{\mathbb{Z}}
\usepackage{amsmath, amssymb, bbm, xspace}

\DeclareMathOperator*{\st}{subject\;to}

\def\diag{\mathop{\hbox{\rm diag}}}

\def\spose#1{\hbox to 0pt{#1\hss}}

\def\text #1{\hbox{\quad#1\quad}}

\def\nthinsp{\mskip -2   mu}

\def\superstar{^{\raise 0.5pt\hbox{$\nthinsp *$}}}
\def\SUPERSTAR{^{\raise 0.5pt\hbox{$*$}}}
\def\lamstarT {\lambda^{\raise 0.5pt\hbox{$\nthinsp *$}T}}

\def\Lscr{{\cal L}}

\def\hbar{\skew{4.2}\bar h}

		\def\bkE{{\rm I\kern-.17em E}}
		\def\bkE{\mathbb{E}}
		\def\bk1{{\rm 1\kern-.17em l}}
		\def\bkD{{\rm I\kern-.17em D}}
		\def\bkR{{\rm I\kern-.17em R}}
		\def\bkP{{\rm I\kern-.17em P}}
		\def\bkY{{\bf \kern-.17em Y}}
		\def\bkZ{{\bf \kern-.17em Z}}

		\def\beq{\begin{eqnarray}}
		\def\bc{\begin{center}}
		\def\be{\begin{enumerate}}
		\def\bi{\begin{itemize}}
		\def\bs{\begin{small}}
		\def\bS{\begin{slide}}
		\def\ec{\end{center}}
		\def\ee{\end{enumerate}}
		\def\ei{\end{itemize}}
		\def\es{\end{small}}
		\def\eS{\end{slide}}
		\def\eeq{\end{eqnarray}}
		\def\qed{\quad \vrule height7.5pt width4.17pt depth0pt}

	\def\cp2problem#1#2#3#4{\fbox
		 {\begin{tabular*}{0.9\textwidth}
			{@{}l@{\extracolsep{\fill}}l@{\extracolsep{6pt}}l@{\extracolsep{\fill}}c@{}}
				#1 & & $#4 $ 
			\end{tabular*}}}

		\renewcommand{\emph}[1]{\textbf{#1}}

		\def\bk1{{\rm 1\kern-.17em l}}
		\def\bkD{{\rm I\kern-.17em D}}
		\def\bkR{{\rm I\kern-.17em R}}
		\def\bkP{{\rm I\kern-.17em P}}
		
		\def\bkZ{{\bf{Z}}}

\newcommand {\beeq}[1]{\begin{equation}\label{#1}}
\newcommand {\eeeq}{\end{equation}}
\newcommand {\bea}{\begin{eqnarray}}
\newcommand {\eea}{\end{eqnarray}}

\def\texitem#1{\par\smallskip\noindent\hangindent 25pt
               \hbox to 25pt {\hss #1 ~}\ignorespaces}

\def\st{\mbox{subject to}}

\def\sjw#1{\footnote{SJW: #1}}

\allowdisplaybreaks
\begin{document}

\title{Complexity of Proximal augmented Lagrangian for nonconvex optimization with nonlinear equality constraints%\thanks{Grants or other notes
%about the article that should go on the front page should be
%placed here. General acknowledgments should be placed at the end of the article.}
}
%\subtitle{Do you have a subtitle?\\ If so, write it here}

\titlerunning{Proximal augmented Lagrangian for nonconvex equality constrained problems}        % if too long for running head

\author{Yue Xie         \and
        Stephen J. Wright %etc.
}

%\authorrunning{Short form of author list} % if too long for running head

\institute{Y. Xie \at
              Wisconsin Institute for discovery, University of Wisconsin, 330 N. Orchard St., Madison, WI
53715. \\
              %Tel.: +123-45-678910\\
              %Fax: +123-45-678910\\
              \email{xie86@wisc.edu}           %  \\
%             \emph{Present address:} of F. Author  %  if needed
           \and
           S. J. Wright \at
              Computer Sciences Department, University of Wisconsin, 1210 W. Dayton St., Madison, WI
53706.\\
          \email{swright@cs.wisc.edu}
}

\date{Received: date / Accepted: date}
% The correct dates will be entered by the editor

\maketitle

\begin{abstract}
We analyze worst-case complexity of a Proximal augmented Lagrangian
(Proximal AL) framework for nonconvex optimization with nonlinear
equality constraints.  When an approximate first-order
(second-order) optimal point is obtained in the subproblem, an
$\epsilon$ first-order (second-order) optimal point for the original
problem can be guaranteed within $\mathcal{O}(1/ \epsilon^{2 - \eta})$
outer iterations (where $\eta$ is a user-defined parameter with
$\eta\in[0,2]$ for the first-order result and $\eta \in [1,2]$ for the
second-order result) when the proximal term coefficient $\beta$ and
penalty parameter $\rho$ satisfy $\beta = \mathcal{O}(\epsilon^\eta)$
and $\rho = \Omega (1/\epsilon^\eta)$, respectively. We also
investigate the total iteration complexity and operation complexity
when a Newton-conjugate-gradient algorithm is used to solve the
subproblems. \Rev{Finally, we discuss an adaptive scheme for determining a value of the parameter $\rho$ that satisfies the requirements of the analysis.}
%Preliminary numerical results support
%our findings and demonstrate efficiency of this traditional method on
%the dictionary learning problem.
% \PACS{PACS code1 \and PACS code2 \and more}
\end{abstract}

\section{Introduction}

Nonconvex optimization problems with nonlinear equality constraints
are common in some areas, including matrix optimization and machine
learning, where such requirements as normalization, orthogonality, or
consensus must be satisfied. Relevant problems include dictionary
learning \cite{7148922}, distributed optimization
\cite{pmlr-v70-hong17a}, and spherical PCA \cite{Liu19a}. We consider
the formulation
\begin{equation}\label{eqcons-opt}
\min \,  f(x) \quad \st \quad c(x)=0,
\end{equation}
where $f: \bR^n \rightarrow \bR $, $c(x) \triangleq
(c_1(x),\dotsc,c_m(x))^T$, $c_i: \bR^n \rightarrow \bR$, $i = 1,
2,\dotsc, m$, and all functions are twice continuously
differentiable.

We have the following definitions related to points that satisfy
approximate first- and second-order optimality coniditions for
\eqref{eqcons-opt}. (Here and throughout, $\| \cdot \|$ denotes
  the Euclidean norm of a vector.)

\begin{definition}[$\epsilon$-1o] \label{def:1o}
  We say that $x$ is an $\epsilon$-1o solution of \eqref{eqcons-opt}
  if there exists $\lambda \in \bR^m$ such that 
\begin{align*}
\| \nabla f(x) + \nabla c(x) \lambda \| \leq \epsilon, \quad \| c(x) \| \leq \epsilon.
\end{align*}
\end{definition}

\begin{definition}[$\epsilon$-2o] \label{Def: epsilonKKT2}
We say that $x$ is an $\epsilon$-2o solution of \eqref{eqcons-opt} if
there exists $\lambda \in \bR^{ m }$ such that:
\begin{subequations} \label{eq:eKKT}
\begin{align}
\label{ineq1:epsilonKKT2}
& \| \nabla f(x) + \nabla c(x) \lambda  \| \leq \epsilon, \quad \| c(x) \|  \leq \epsilon, \\
\label{ineq3:epsilonKKT2}
& d^T \left( \nabla^2 f(x) + \sum_{i = 1}^m \lambda_i \nabla^2 c_i(x) \right) d \geq -\epsilon \| d \|^2,
\end{align}
\end{subequations}
for any $d \in S(x) \triangleq \{ d \in \bR^n \mid \nabla c(x)^T d = 0 \}$.
\end{definition}

These definitions are consistent with those of $\epsilon$-KKT and
$\epsilon$-KKT2 in \cite{birgin2018augmented}, and similar to those of
\cite{Haeser2018}, differing only in choice of norm and use of $\|
c(x) \| \le \epsilon$ rather than $c(x)=0$. The following theorem is
implied by several results in \cite{And17a} and
\cite{birgin2018augmented}, which consider a larger class of problem
than \eqref{eqcons-opt}. (A proof tailored to \eqref{eqcons-opt} is
supplied in the Appendix.)
\begin{theorem} \label{OptCon}
If $x^*$ is an local minimizer of \eqref{eqcons-opt}, then there
exists $\epsilon_k \rightarrow 0^+$ and $x_k \rightarrow x^*$ such
that $x_k$ is $\epsilon_k$-2o, thus $\epsilon_k$-1o.
\end{theorem}

Theorem~\ref{OptCon} states that being the limit of a sequence of
points satisfying Definition~\ref{def:1o} or Definition~\ref{Def:
  epsilonKKT2} for a decreasing sequence of $\epsilon$ is \Rev{a
  necessary} condition of a local minimizer. When certain
  constraint qualifications hold, a converse of this result is also
  true: $x^*$ satisfies first-order (KKT) conditions when $x_k$ is
  $\epsilon_k$-1o and second-order conditions when $x_k$ is
  $\epsilon_k$-2o (See \cite{And17a,And16a}). These observations
justify our strategy of seeking points that satisfy
Definition~\ref{def:1o} or \ref{Def: epsilonKKT2}.
% \sjw{I edited this paragraph to define ``CQ'' which has not yet
% appeared at this point, and to fix the grammar. But I don't
% understand what it is trying to say. It is saying that a kind
% of converse of Theorem~\ref{OptCon} holds when a CQ is satisfied?\\
% Yue: Yeah. A kind of converse of Thm 1 holds with $x^*$ being a KKT
% point instead of a local minimizer.}

The augmented Lagrangian (AL) framework is a penalty-type algorithm
for solving \eqref{eqcons-opt}, originating with
Hestenes~\cite{Hestenes1969} and Powell~\cite{MR0272403}.  Rockafellar
proposed in \cite{Roc76a} the proximal version of
this method, which has both theoretical and practical advantages. The
monograph \cite{bertsekas2014constrained} summarizes development of this
method during the 1970s, when it was known as the ``method of
multipliers''. Interest in the algorithm has resurfaced in recent
years because of its connection to ADMM \cite{boyd2011distributed}.
%% which is based on AL.

The augmented Lagrangian of \eqref{eqcons-opt} is defined as:
\[
\sL_\rho (x,\lambda) \triangleq f(x) + \sum_{i = 1}^m \lambda_i c_i(x) + \frac{\rho}{2} \sum_{i = 1}^m \Rev{| c_i(x) |^2} =
f(x) + \lambda^T c(x) + \frac{\rho}{2} \| c(x) \|^2,
\]
where $\lambda \triangleq (\lambda_1, \dotsc,\lambda_m)^T$. The
(ordinary) Lagrangian of \eqref{eqcons-opt} is $\sL_0(x, \lambda)$.

\subsection{Complexity measures}

In this paper, we discuss measures of worst-case complexity for
finding points that satisfy Definitions~\ref{def:1o} and \ref{Def:
  epsilonKKT2}. Since our method has two nested loops --- an outer
loop for the Proximal AL procedure, and an inner loop for solving the
subproblems --- we consider the following measures of complexity.
\begin{itemize}
  \item {\em Outer iteration complexity,} which corresponds to the
    number of outer-loop iterations of Proximal AL or some other
    framework;
  \item {\em Total iteration complexity,} which measures the total
    number of iterations of the inner-loop procedure that is required
    to find points satisfying approximate optimality of the subproblems;
  \item {\em Operation complexity,} which measures the number of some
    unit operation (in our case, computation of a matrix-vector
    product involving the Hessian of the Proximal augmented Lagrangian) required to find
    approximately optimal points.
\end{itemize}
We also use the term ``total iteration complexity'' in connection with
algorithms that have only one main loop, such as those whose
complexities are shown in Table~\ref{tab: complex.}.

We prove results for all three types of complexity for the Proximal AL
procedure, where the inner-loop procedure is a
Newton-conjugate-gradient (Newton-CG) algorithm for the unconstrained
nonconvex subproblems. Details are given in
Section~\ref{sec:contributions}.

\begin{algorithm}
\caption{Augmented Lagrangian (AL)}\label{Alg: AL}
\be
\item[0.]  Initialize $x_0$, $\lambda_0$ and $\rho_0 > 0$, $\Lambda \triangleq [\Lambda_{\min}, \Lambda_{\max}] $, $\tau \in (0,1)$, $\gamma > 1$; Set $k \leftarrow 0$;
\item[1.] Update $x_k$:  Find approximate solution $x_{k+1}$ to $\min_x \,  \Lscr_{\rho_k}(x,\lambda_k)$;
\item[2.] Update $\lambda_k$: $ \lambda_{k+1} \leftarrow P_{\Lambda} ( \lambda_k + \rho_k c(x_{k+1}) )$;
\item[3.] Update $\rho_k$: if $k = 0$ or $\| c(x_{k+1}) \|_{\infty}
  \leq \tau \| c(x_k) \|_{\infty}$, set $\rho_{k+1} = \rho_k$;
  otherwise, set $\rho_{k+1} = \gamma \rho_k$;
\item[4.] If termination criterion is satisfied, STOP; otherwise, $k \leftarrow k + 1$ and return to Step 1.
\ee
\end{algorithm}

\subsection{Related work}\label{subsec: relate.work.}

\paragraph{AL for nonconvex optimization.}  
%\sjwcomment{What is the deal with $[\lambda_{\min}, \lambda_{\max}]$?
%  Are they really needed? What conditions are required on them?
%  \yx{YX: they are needed in Birgin et al.'s work
%  \cite{birgin2018augmented}, otherwise the multiplier sequence may
%  goes to infinity.} }
We consider first the basic augmented Lagrangian framework outlined in
Algorithm~\ref{Alg: AL}. When $f$ is a nonconvex function, convergence
of the augmented Lagrangian framework has been studied in
\cite{Birgin2010,birgin2018augmented}, with many variants described
in
\cite{And08a,Andreani2010,And19a,And18b,Curtis2015}. In
\cite{birgin2018augmented}, Algorithm~\ref{Alg: AL} is investigated
and generalized for a larger class of problems, showing in particular
that if $x_{k+1}$ is a first-order (second-order) approximate solution
of the subproblem, with error driven to $0$ as $k \to \infty$,
% \sjw{Does this mean that the subproblem solutions are approximate?
%   If so we should say ``approximate solution'' not
%   ``solution''. Yue: Corrected.}
then every feasible limit point is an approximate first-order
(second-order) KKT point of the original problem. In
\cite{Birgin2010}, it is shown that when the subproblem in
Algorithm~\ref{Alg: AL} is solved to approximate global optimality
with error approaching $0$, the limit point is feasible and is a
global solution of the original problem.

There are few results in the literature on outer iteration complexity
in the nonconvex setting. Some quite recent results appear in \cite{Yuan.arXiv2019,birgin2020complexity}. In \cite{Yuan.arXiv2019},
the authors apply a general version of augmented Lagrangian to
nonconvex optimization with both equality and inequality
constraints. With an aggressive updating rule for the penalty
parameter, they show that the algorithm obtains an approximate KKT
point (whose exact definition is complicated, but similar to our
definition of $\epsilon$-1o optimality when only equality constraints
are present) within $\sO(\epsilon^{-2/(\alpha-1)})$ outer-loop
iterations, where $\alpha > 1$ is an algorithmic parameter. This
complexity is improved to $\sO(|\log \epsilon|)$ when boundedness of
the sequence of penalty parameters is assumed. Total iteration
complexity measures are obtained for the case of linear equality
constraints when the subproblem is solved with a $p$-order method ($p
\ge 2$). In \cite{birgin2020complexity}, the authors study an
augmented Lagrangian framework named  \texttt{ALGENCAN} to problems with
equality and inequality constraints. An $\epsilon$-accurate first-order point (whose precise definition is again similar to our
$\epsilon$-1o optimality in the case of equality constraints only) is
obtained in $\sO( | \log \epsilon | )$ outer iterations when the
penalty parameters are bounded. The practicality of the assumption of
bounded penalty parameters in these two works \cite{Yuan.arXiv2019,birgin2020complexity} is open to question,
since the use of an increasing sequence of penalty parameters is
critical to both approaches, and there is no clear prior reason why
the sequence should be bounded\footnote{\Rev{Circumstances under which the penalty parameter sequence of \texttt{ALGENCAN} is bounded are discussed in \cite[Section~5]{And08a}.}}. 

\paragraph{Proximal AL for nonconvex optimization: Linear equality constraints.}  
The Proximal augmented Lagrangian framework, with fixed positive
parameters $\rho$ and $\beta$, is shown in Algorithm~\ref{Alg:
  Prox-PDA}.

\begin{algorithm}
\caption{Proximal augmented Lagrangian (Proximal AL)}\label{Alg: Prox-PDA}
\be
\item[0.]  Initialize $x_0$, $\lambda_0$ and $\rho > 0$, $\beta > 0$; Set $k \leftarrow 0$;
\item[1.] Update $x_k$:  Find approximate solution $x_{k+1}$ to $\min_x \, \Lscr_\rho(x,\lambda_k) + \frac{\beta}{2} \| x - x_k \|^2$;
\item[2.] Update $\lambda_k$: $ \lambda_{k+1} \leftarrow \lambda_k + \rho c(x_{k+1})$;
\item[3.] If termination criterion is satisfied, STOP; otherwise, $k \leftarrow k + 1$ and return to Step 1.
\ee
\end{algorithm}

For this proximal version, in the case of {\em linear} constraints
$c(\cdot)$, outer iteration complexity results become accessible in
the nonconvex regime
\cite{Hajinezhad2019,pmlr-v70-hong17a,Jiang2019,doi:10.1137/19M1242276}. The
paper \cite{pmlr-v70-hong17a} analyzes the outer iteration complexity
of this approach (there named ``proximal primal dual algorithm
(Prox-PDA)'') to obtain a first-order optimal point, choosing a
special proximal term to make each subproblem strongly convex and
suitable for distributed implementation. An outer iteration complexity
estimate of $\sO(\epsilon^{-1})$ is proved for an $\sqrt{\epsilon}$-1o
point. This result is consistent with our results in this paper when
the choice of $\beta$ and $\rho$ is independent of $\epsilon$ and
$c(x)$ is linear.

% \sjw{We just said our results are consistent with theirs, now we say
% ``we improve this complexity.'' Which is it? Do we improve it by
% generalizing the choices of $\beta$ and $\rho$? We should be more
% specific. Do we improve by allowing nonlinear $c$ as
% well? \\
% Yue: Sorry for the confusion. I just mention the consistency part
% and try to be more specific.}

%   We improve this complexity, as well as deriving complexity results
%   for approximate second-order optimality, by allowing $\beta$ and
%   $\rho$ to be dependent on $\epsilon$.

The paper \cite{Hajinezhad2019} proposes a ``perturbed proximal primal
dual algorithm,'' a variant of Algorithm~\ref{Alg: Prox-PDA}, to
obtain outer iteration complexity results for a problem class where
the objective function may be nonconvex and nonsmooth. In particular,
an outer iteration complexity of $\sO(\epsilon^{-2})$ is required to
obtain $\epsilon$-stationary solution, where the latter term is
defined in a way that suits that problem class. A modified inexact
Proximal AL method is investigated in
\cite{doi:10.1137/19M1242276}. Here, an exponentially weighted average of
previous updates is used as the anchor point in the proximal term, total iteration complexity of $\sO(\epsilon^{-2})$ to locate an $\epsilon$ stationary point similar to $\epsilon$-1o is derived and
a certain kind of linear convergence is proved for quadratic
programming (QP). The paper \cite{Jiang2019} derives outer iteration
complexity of $\sO(\epsilon^{-2})$ for a proximal ADMM procedure to
find an $\epsilon$ stationary solution defined for their problem
class.

%% To repeat, $c(\cdot)$ is assumed to be linear in all these works.

To our knowledge, outer iteration complexity of Proximal AL in the
case of {\em nonlinear} $c(x)$
% \footnote{By ``nonlinear $c(x)$'', we mean that the nonlinear
% constraint $c(x) = 0$ will be penalized in the augmented Lagrangian
% function instead of being enforced explicitly in the subproblem.}
and \Rev{its complexity} for convergence to second-order optimal
points have not yet been studied.

\paragraph{Complexity for constrained nonconvex optimization.}
For constrained nonconvex optimization, worst-case total iteration
complexity results of various algorithms to find $\epsilon$-perturbed
first-order and second-order optimal points have been obtained in
recent years. If only first-derivative information is used, total
iteration complexity to obtain an $\epsilon$-accurate first-order
optimal point may be $\mathcal{O}( \epsilon^{-2} )$
\cite{Bian2015,Haeser2018,2018arXiv181002024N}. If Hessian information
is used (either explicitly or via Hessian-vector products), total
iteration complexity for an $\epsilon$-accurate first-order point can
be improved to $\mathcal{O}( \epsilon^{-3/2} )$
\cite{Bian2015,Haeser2018,10.1093/imanum/drz074}, while the total
iteration complexity to obtain an $\epsilon$-accurate second-order
point is typically $\mathcal{O}( \epsilon^{-3} )$
\cite{Bian2015,Haeser2018,2018arXiv181002024N,10.1093/imanum/drz074}. More
  details about these results can be found in Table~\ref{tab:
    complex.}.

%% \rev{\em Here add some discussion about the CGT work on exact penalty
%%   function and 2-phase methods, and Birgin et al on 2-phase
%%   methods. Say in particular that (a) they are about a different
%%   algorithm; our goal here is to analyze the more-or-less standard
%%   version of augmented Lagrangian; (b) they get only approx
%%   first-order optimality results, whereas our approach provides a
%%   pathway to second-order results.}

\Rev{Other approaches focus on nonlinear equality constraints and
   seek evaluation complexity bounds (``Evaluation
    complexity'' refers to the number of evaluations of $f$ and $c$
    and their derivatives required, and corresponds roughly to our
    ``total iteration complexity''.) for approximate first-order optimality.  An algorithm
 based on linear approximation of the exact penalty function for
 \eqref{eqcons-opt} is described in \cite{Car11a}, and attains a
 worst-case evaluation complexity of $\sO(\epsilon^{-5})$ \rev{by
   using only function and gradient information.} Two-phase
 approaches, which first seek an approximately feasible point by
 minimizing the nonlinear least-squares objective $\| c(x) \|_2^2$ (or equivalently $\| c(x) \|$),
%% (\rev{or $\| c(x) \|$}),
and then apply a target-chasing method to find an approximate first-order
 point for \eqref{eqcons-opt}, are described in
 \cite{doi:10.1137/120869687,Car14a}. (See Table~\ref{tab:
     complex.}.) Extensions of these techniques to approximate
 second-order optimality is not straightforward; most such efforts
 focus on special cases such as convex constraints. A recent work that
 tackles the general case is \cite{Car19a}, which again considers the
 two-phase approach and searches for approximate first-, second-, and
 third-order critical points. Specific definitions of the
   critical points are less interpretable; we do not show them in
   Table~\ref{tab: complex.}. They are related to scaled KKT
   conditions for the first order point, and to local optimality with
   tolerance of a function of $\epsilon$ for second and third order
   points.}

\begin{table}
\scriptsize
  \caption{Total iteration or evaluation complexity estimates for constrained nonconvex optimization procedures. Here
    $X=\diag(x)$ and $\bar{X} = \diag(\min\{ x, {\bf 1} \})$. $\tilde{\sO}$ represents $\sO$ with logarithm factors hidden. Gradient or Hessian in parenthesis means that the algorithm uses only gradient or both gradient and Hessian information, respectively. $p$th derivative means that the algorithm needs to evaluate function derivatives up to $p$th order.}
  \label{tab: complex.}
  \centering
  \begin{tabular}{llcl}
  \toprule
  Point type%\tablefootnote{In definitions of the points listed in Table~\ref{tab: complex.}, note that $X \triangleq \diag(x) $, $\bar{X} \triangleq \diag(\min\{ x, {\bf 1} \})$, $\tilde{\sO}$ represents $\sO$ and hides the logarithm terms.}
  & Complexity &  Constraint type & Lit. \\
  \hline
  $
  \begin{cases} 
  | [ X \nabla f(x) ]_i | \le \epsilon, & \mbox{if } x_i < (1 - \epsilon/2)b_i \\ 
  [\nabla f(x)]_i \le \epsilon, & \mbox{if } x_i \ge (1 - \epsilon/2)b_i
  \end{cases}
  $ 
   & $ \begin{array}{c}
   
   \end{array}
   \sO(\epsilon^{-2})$ (gradient)
  & $0 \le x \le b$  & \cite{Bian2015}  \\
  $ \| X \nabla f(x) \|_\infty \le \epsilon$, $X \nabla^2 f(x) X \succeq - \sqrt{\epsilon}I_n$ & $\sO(\epsilon^{-3/2})$ (Hessian) & $x \ge 0$  & \cite{Bian2015} \\
  \hline
  $ 
  \begin{cases}  
  Ax = b, \; x > 0, \; \nabla f(x) + A^T \lambda \ge -\epsilon {\bf 1} \\
  \| X ( \nabla f(x) + A^T \lambda ) \|_\infty \le \epsilon
  \end{cases}
  $ & $\sO(\epsilon^{-2})$ (gradient) & $ Ax = b, x \ge 0$ & \cite{Haeser2018} \\
  $ 
  \begin{cases}  
  Ax = b, \; x > 0, \; \nabla f(x) + A^T \lambda \ge -\epsilon {\bf 1} \\
  \| X ( \nabla f(x) + A^T \lambda ) \|_\infty \le \epsilon \\
  d^T ( X \nabla^2f(x) X + \sqrt{\epsilon} I ) d \ge 0, \\
  \forall d \in \{ d \mid AXd = 0 \}
  \end{cases}
  $ & $\sO(\epsilon^{-3/2})$ (Hessian) &  $ Ax = b, x \ge 0$ & \cite{Haeser2018} \\
  \hline
% $ \left| \left\{ \begin{array}{l}
% \min \left\langle \nabla f(x), s \right\rangle \\ 
% \ s.t.\  x + s \in \sF, \| s \| \le 1
% \end{array}  \right\} \right| \le \epsilon$ &  $\sO(\epsilon^{-2})$ & iteration/evaluation(of gradient) & $x \in \sF$. $\sF$ is closed and convex. & \cite{2018arXiv181002024N} \\
 $ 
\begin{cases} 
 \left| \left\{ \begin{array}{l}
 \min_s \left\langle \nabla f(x), s \right\rangle \\ 
 \ s.t.\  x + s \in \sF, \; \| s \| \le 1
 \end{array}  \right\} \right| \le \epsilon_g \\
 \left| \left\{ \begin{array}{l}
 \min_d d^T \nabla^2 f(x) d \\ 
 \ s.t.\ x + d \in \sF, \; \| d \| \le 1, \\
 \quad \quad \left\langle \nabla f(x), d \right\rangle \le 0
 \end{array} \right\} \right| \le \epsilon_H
\end{cases} 
 $ &  $\begin{array}{c}
 \sO(\max\{\epsilon_g^{-2}, \epsilon_H^{-3} \}) \\
 \mbox{(Hessian)}
 \end{array}$ &  
 $
\begin{array}{c}
 x \in \sF, \\
 \mbox{$\sF$ is closed}\\
 \mbox{and convex}
 \end{array} 
 $
  & \cite{2018arXiv181002024N} \\
 \hline
 $
\begin{cases}
& x > 0, \; \nabla f(x) \ge - \epsilon {\bf 1}, \; \| \bar{X} \nabla f(x) \|_\infty \le \epsilon,  \\
&  \bar{X} \nabla^2 f(x) \bar{X} \succeq -\sqrt{\epsilon} I 
\end{cases} 
$
& $\tilde{\sO}(\epsilon^{-3/2})$ (Hessian) & $x \ge 0$ & \cite{10.1093/imanum/drz074} \\
\hline
$
\begin{array}{c}
 \| \nabla f(x) + \nabla c(x) \lambda \| \le \epsilon, \| c(x) \| \le \epsilon, \mbox{ or} \\ \mbox{$x$ is an approximate critical point of $\| c(x) \|$}
\end{array} 
$
& $ \sO(\epsilon^{-5}) $ (gradient)
& $c(x) = 0$ & \cite{Car11a} \\
\hline
$
\begin{array}{c}
\| c(x) \| \le \epsilon_p, \ \| \nabla f(x) + \nabla c(x) \lambda \|  \le \epsilon_d \| (\lambda,1) \| \\
\mbox{or } \| \nabla c(x) c(x) \| \le \epsilon_d \| c(x) \|
\end{array}
$ & $ \begin{array}{c}
\sO(\epsilon_d^{-3/2} \epsilon_p^{-1/2}) \\
\mbox{(Hessian)}
\end{array} $ & c(x) = 0 & \cite{doi:10.1137/120869687} \\
\hline 
%$\begin{array}{c}
%\| \nabla f(x) + \nabla c_{\mathcal{E}}(x) \lambda + \nabla c_{\mathcal{I}}(x) \mu \| \le \epsilon \\
%\| c_{\mathcal{E}}(x) \|^2 + \| c_{\mathcal{I}}(x)_+ \|^2 \le \epsilon^2
%\end{array} 
%$ &
%$
%\begin{array}{c}
%\sO\left( \epsilon^{1-\frac{2(p+1)}{p}} \right) \\
%\mbox{ ($p$th derivative) }
%\end{array}
%$ & $ \begin{array}{c}
%c_{\mathcal{E}}(x) = 0 \\
%c_{\mathcal{I}}(x) \le 0
%\end{array} $ & \cite{Bir16c} \\
%\hline
$
\begin{array}{c}
\| \nabla f(x) + \nabla c(x) \lambda \| \le \epsilon$, $ \| c(x) \| \le \epsilon, \mbox{ or} \\
\| \nabla c(x) \mu \| \le \epsilon, \| c(x) \| \ge \kappa \epsilon.
\end{array}
$ & $\sO(\epsilon^{-2}) $ (gradient) & $c(x) = 0$ & \cite{Car14a} \\
\hline
$
\begin{array}{c}
\mbox{x is $\epsilon$ approximate first order critical point } \\
\mbox{of the constrained problem or of $\| c(x) \|$}
\end{array}$ & $
\begin{array}{c}
\sO(\epsilon^{-(p+2)/p}) \\
\mbox{($p$th derivative)}
\end{array}
$ & $
\begin{array}{c}
c(x) = 0, x \in \sF \\
\sF \mbox{ is closed} \\
\mbox{and convex}
\end{array}
$ & \cite{Car19a} \\
$
\begin{array}{c}
\mbox{x is $\epsilon$ approximate $q$th order critical point } \\
\mbox{of the constrained problem or of $\| c(x) \|$} \\
q = 1,2,3.
\end{array}
$ & $
\begin{array}{c}
\sO(\epsilon^{-2q - 1}) \\
\mbox{($q$th derivative)}
\end{array}
$ & $
\begin{array}{c}
c(x) = 0, x \in \sF \\
\sF \mbox{ is closed} \\
\mbox{and convex}
\end{array}
$ & \cite{Car19a} \\
\bottomrule
\end{tabular}
\end{table}
  
% Notably, the known complexity of
%Proximal AL and Proximal ADMM to obtain $\epsilon$-1o of nonconvex
%optimization with linear equality constraints is also
%$\sO(\epsilon^{-2})$ \cite{hong2016decomposing, Jiang2019}.
%\sjwcomment{Are the definitions of $\epsilon$ first-order and
%  second-order points here the same as those we use in this paper? \yx{YX: Not exactly. Actually they are not even the same with each other. Translation is needed. I just list the mainstream complexity results. Maybe we should point this out.} }

%\sjwcomment{First you say that both us and them use major
%  iterations to measure complexity, then you say that our subproblems
%  are more expensive. Are both things true? What makes their
%  subproblems cheaper?}

%If analysis of overall computational complexity is done then
%  comparison will make more sense and we will leave this work to the
%  future.

  \subsection{Contributions} \label{sec:contributions}
  
We apply the Proximal AL framework of Algorithm~\ref{Alg: Prox-PDA} to
\eqref{eqcons-opt} for nonlinear constraints $c(x)$. Recalling
Definitions~\ref{def:1o} and \ref{Def: epsilonKKT2} of approximately
optimal points, we show that when approximate first-order
(second-order) optimality is attained in the subproblems, the outer
iteration complexity to obtain an $\epsilon$-1o ($\epsilon$-2o) point
is $\mathcal{O}(1/\epsilon^{2 - \eta})$ if we let $\beta =
\sO(\epsilon^\eta)$ and $\rho = \Omega(1/\epsilon^\eta)$, where
$\eta$ is a user-defined parameter with $\eta\in[0,2]$ for the
first-order result and $\eta \in [1,2]$ for the second-order result.
We require uniform boundedness and full rank of the constraint
Jacobian on a certain bounded level set, and show that the primal and
dual sequence of Proximal AL is bounded and the limit point satisfies
first-order KKT conditions.
    
We also derive total iteration complexity of the algorithm when the
Newton-CG algorithm of \cite{Royer2019} is used to solve the
subproblem at each iteration of Algorithm~\ref{Alg: Prox-PDA}. The
operation complexity for this overall procedure is also described,
taking as unit operation the computation of a Hessian-vector
product. When $c(x)$ is linear and $\eta = 2$, the total iteration
complexity matches the known results in literature for second-order
algorithms: $\sO(\epsilon^{-3/2})$ for an $\epsilon$-1o point and
$\sO(\epsilon^{-3})$ for an $\epsilon$-2o point.

\Rev{Finally, we present a scheme for determining the algorithmic parameter
$\rho$ adaptively, by increasing it until convegence to an
approximately-optimal point is identified within the expected number
of iterations.}

%Preliminary numerical experiments, reported in Section~\ref{sec: Num},
%are consistent with the theoretical findings, in the ways in which the
%outer iteration counts depend on parameters $\rho$ and $\beta$.  We
%also show the efficiency of Proximal AL on dictionary learning.
 %% \sjw{I notice here that you don't show the general form of
 %%  evaluation complexity here, involving $\eta$, just give the
 %%  optimal results for $\eta=2$. This seems inconsistent with the
 %%  ``major iteration complexity'' results summarized earlier. Yue: I
 %%  specified $\eta = 2$. The full results seem too much here.}

\subsection{Organization}

In Section~\ref{sec: Prelim}, we list the notations and main
assumptions used in the paper. We discuss outer iteration
  complexity of Proximal AL in Section~\ref{sec: out.comp.}. Total
  iteration complexity and operation complexity are derived in
  Section~\ref{sec: total.iter.comp.}. A framework for determining the
  parameter $\rho$ in Proximal AL is proposed in Section~\ref{sec:
    varyrho}. We summarize the paper and discuss future work in
Section~\ref{sec: conclusion}. Most proofs appear in the main body of
the paper; some elementary results are proved in the Appendix.

%{\bf Augmented Lagrangian and Proximal augmented Lagrangian}
%\be
%\item Who proposes AL and PAL. 
%\item What is known about AL and PAL. convex and nonconvex.
%\item The gap.
%\ee
%{\bf First- and second-order complexity of nonconvex optimization}
%Explain known complexity results for (un-)constrained optimization. Explain that the complexity considered here is for major iterations. (Maybe relate it to some known results.)
%
%We attempt to solve the equality constrained program \eqref{eqcons-opt} through augmented Lagrangian framework as follows:
%\begin{algorithm}
%\caption{Augmented Lagrangian Framework}\label{Alg: AL}
%\be
%\item[0.]  Initialize $x_0$, $\lambda_0$ and $\rho_0 > 0$. Set $k \leftarrow 0$;
%\item[1.]  Choose a start point $x_k^0$, and find $x_{k+1}$ such that $\sL_{\rho_k} (x_{k+1},\lambda_k)$ satisfies second-order necessary conditions at $x_{k+1}$, i.e.
%$$ 
%\nabla_x \sL_{\rho_k} (x_{k+1}, \lambda_k) = 0 \mbox{ and } \nabla^2_{xx} \sL_{\rho_k} (x_{k+1},\lambda_k) \succeq 0;
%$$
%\item[2.] Update $\lambda_k$: $ \lambda_{k+1} \leftarrow \lambda_k + \rho_k c(x_{k+1})$;
%\item[3.] Update $\rho_k$;
%\item[4.] If termination criteria is satisfied, STOP; otherwise, $k \leftarrow k + 1$ and return to Step 1.
%\ee
%\end{algorithm}

%\subsection{Literature review}
%\cite{hong2016decomposing}
  
\section{Preliminaries} \label{sec: Prelim}

\paragraph{Notation.}
%% For any vector $v \in \bR$ and matrix $M \in \bR^{n \times n}$, $\|
%% v \|_M \triangleq \sqrt{v^TMv}$, where $M \succeq 0$.
We use $\| \cdot \|$ to denote the Euclidean norm of a vector and $\|
\cdot \|_2$ to denote the operator $2$-norm of a matrix. For a given
matrix $H$, we denote by $\sigma_{\min}(H)$ its minimal singular
value and by $\lambda_{\min}(H)$ its minimal eigenvalue. We denote steps in $x$ and $\lambda$ as follows:
\begin{equation} \label{eq:DxDl}
  \Delta x_{k+1}
\triangleq x_{k+1} - x_k, \quad \Delta \lambda_{k+1} \triangleq
\lambda_{k+1} - \lambda_k.
\end{equation}

In estimating complexities, we use order notation $\sO(\cdot)$ in the
usual sense, and $\tilde{\sO}$ to hide factors that are logarithmic in
the arguments. We use $\beta (\alpha) = \Omega(\gamma (\alpha ) )$ (where
$\beta(\alpha)$ and $\gamma(\alpha)$ are both positive) to indicate that
$\beta (\alpha)/ \gamma(\alpha)$ is bounded below by a positive real number for all $\alpha$ sufficiently small.

%Let $\sD \triangleq \{ x \mid f(x) < +\infty \}$.

%\yx{Add a paragraph to motivate the Proximal augmented Lagrangian framework.}

%In order to obtain global convergence rate, we may utilize a proximal variant of the augmented Lagrangian.

\paragraph{Assumptions.}

%We assume throughout that $\nabla f$ is Lipschitz continuous over
%$\bR^n$, that is, there exists a constant $L_f$ such that
%\begin{align} \label{eq:L}
%\| \nabla f(x) - \nabla f(y) \| \leq L_f \| x - y \|, \quad \mbox{for all $x, y \in \bR^n$}.
%\end{align}

%\sjwcomment{$\sD$ is not defined before is it? At the start we simply
%use $\sD = R^n$. Do we really want to bring in $\sD$ here? In fact at
%the start of the paper we made a different set of assumptions
%altogether. \yx{YX: $\sD$ is the domain of $f$ where it takes real
%value (not $+\infty$), so it's consistent with the beginning. Yet we
%could let $\sD = \bR^n$. }}

The following assumptions are used throughout this work.

\begin{assumption} \label{Ass: Compact sublevel set}
Suppose that there exists $\rho_0 \geq 0$ such that $f(x) +
\frac{\rho_0}{2} \| c(x) \|^2$ has compact level sets, that is, for
all $\alpha \in \bR$, the set
\begin{equation} \label{eq:S0a}
  S_{\alpha}^0 \triangleq \left\{ x \Big| f(x) + \frac{\rho_0}{2} \| c(x) \|^2 \leq
  \alpha \right\}
\end{equation}
is empty or compact.
\end{assumption}
Assumption~\ref{Ass: Compact sublevel set} holds in any of the following cases:
\begin{itemize}
\item[1.] $f(x) + \frac{\rho_0}{2} \| c(x) \|^2$ is coercive for some
  $\rho_0 \ge 0$.
\item[2.] $f(x)$ is strongly convex.
\item[3.] $f(x)$ is bounded below and $c(x)=x^Tx-1$, as occurs in
  orthonormal dictionary learning applications.
\item[4.] $f(x) \triangleq \frac{1}{2} x^T Q
x - p^Tx$, $c(x) \triangleq Ax - b$, $Q$ is positive definite on
$\mbox{null}(A) \triangleq \{ x \mid Ax = 0 \}$.
  \end{itemize}
%% When $f(x)$ is strongly convex, Assumption~\ref{Ass: Compact sublevel
%%   set} holds; if $f(x)$ is lower bounded and $c(x)$ includes
%% constraint such as $q^Tq - 1 = 0$ or $ \| q \| - 1 = 0$ (dictionary
%% learning \eqref{DL}), Assumption~\ref{Ass: Compact sublevel set} also
%% holds. If $f \triangleq \frac{1}{2} x^T Q x - p^Tx$, $c(x) \triangleq
%% Ax - b$, $Q$ is positive definite on $\mbox{null}(A) \triangleq \{ x
%% \mid Ax = 0 \}$, then $f(x) + \frac{\rho_0}{2} \| c(x) \|$ is strongly
%% convex for some $\rho_0 > 0$. Therefore, Assumption~\ref{Ass: Compact
%%   sublevel set} holds.

An immediate consequence of this assumption is the following, 
proof of which appears in the Appendix.
\begin{lemma} \label{lm: Lb}
Suppose that Assumption~\ref{Ass: Compact sublevel set} holds, then $
f(x) + \frac{\rho_0}{2} \| c(x) \|^2 $ is lower bounded.
\end{lemma}
Therefore, Assumption~\ref{Ass: Compact sublevel set} implies
\begin{equation} \label{eq:defL}
\bar L \triangleq \inf_{x \in \bR^n} \left\{ f(x) +
  \frac{\rho_0}{2} \| c(x) \|^2 \right\} > -\infty. 
\end{equation}
We use this definition of $\bar L$ throughout this paper whenever Assumption~\ref{Ass: Compact sublevel set} holds.

The second assumption concerns certain smoothness and nondegeneracy assumptions on $f$ and $c$ over a compact set.
\begin{assumption}\label{Ass: weak}
Given a compact set $\sS \subseteq \bR^n$, there exist positive
constants $M_f$, $M_c$, $\sigma$, $L_c$ such that the following conditions
on functions $f$ and $c$ hold.
\begin{enumerate}
\item[(i)] $\| \nabla f(x) \| \leq M_f$, $ \| \nabla f(x) - \nabla f(y) \| \le L_f \| x - y \| $, for all $x, y \in \sS$.
\item[(ii)] $ \| \nabla c(x) \|_2 \leq M_c $, $\sigma_{\min} (
  \nabla c(x)) \geq \sigma > 0$ for all $x \in \sS$.
\item[(iii)] $ \| \nabla c(x) - \nabla c(y) \|_2 \leq L_c \| x - y
  \|$, for all $ x, y \in \sS. $
\end{enumerate}
\end{assumption}

This assumption may allow a general class of problems; in particular,
(i) holds if $f(x)$ is smooth and $\nabla f(x)$ is locally Lipschitz
continuous on a neighborhood of $\sS$. (ii) holds when $c(x)$ is smooth on a neighborhood of $\sS$ and
satisfies an LICQ condition over $\sS$,
%% \sjw{I didn't quite
%%   understand what you said here previously and I edited it. Is my
%%   statement correct?\\ Yue: Yeah. I saw you deleted MFCQ but that is
%%   ok.}
and (iii) holds if $\nabla c(x)$ is locally Lipschitz continuous
on $\sS$.
%% \sjw{I added this about $x \in S$\\ Yue: That is fine. Actually the
%%   definition of locally Lipschitz continuity usually implicitly
%%   includes every point on the domain. Now I changed it to be more
%%   rigorous. Since the set $\sS$ can be arbitrary, we need
%%   smoothness, locally Lipschitz continuity and LICQ to hold on
%%   $\bR^n$.}

\Rev{
\begin{assumption} \label{Ass: f.ub}
Suppose that $f(x) \le \bar U$ for any $x \in \{ x \mid \| c(x) \| \le 1 \}$.
\end{assumption} 
A sufficient condition for Assumption~\ref{Ass: f.ub} to hold is the
compactness of $\{ x \mid \| c(x) \| \le 1 \}$. This assumption is not
needed if $c(x_0)=0$, that is, the initial point is feasible.}

\section{Outer iteration complexity of Proximal AL} \label{sec: out.comp.}

In this section, we derive the outer iteration complexity of Proximal
AL (Algorithm~\ref{Alg: Prox-PDA}) when the subproblem is solved
inexactly.
% under Assumption~\ref{Ass: Compact sublevel set} and
% Assumption~\ref{Ass: weak}.
We assume that $x_{k+1}$ in Step 1 of Algorithm~\ref{Alg: Prox-PDA}
satisfies the following approximate first-order optimality condition:
\begin{align} \label{eq:1oi}
\nabla_x \Lscr_{\rho} (x_{k+1}, \lambda_k) + \beta ( x_{k+1} - x_k ) = \tr_{k+1}, \quad \mbox{for all $k \ge 0$,}
\end{align}
where $\tr_{k+1}$ is some error vector. We additionally assume that
\begin{align} \label{eq:decr}
\Lscr_{\rho}(x_{k+1}, \lambda_k) + \frac{\beta}{2} \| x_{k+1} - x_k \|^2 \leq \Lscr_{\rho}(x_k, \lambda_k), \quad \mbox{for all $k \ge 0$.}
\end{align}
This condition can be achieved if we choose $x_k$ as the initial point
of the subproblem in Step 1 of Algorithm~\ref{Alg: Prox-PDA}, with
subsequent iterates decreasing the objective of this subproblem.  To
analyze convergence, we use a Lyapunov function defined as follows for
any $k \geq 1$, inspired by \cite{pmlr-v70-hong17a}:
\begin{align} \label{Def: P-func}
P_k \triangleq \sL_\rho(x_k, \lambda_k) + \frac{\beta}{4} \| x_k - x_{k-1} \|^2. 
\end{align}
%By the assumption of solving the subproblem in Step 1, we have the following:
%\begin{align}\label{eq:decr}
%\sL_\rho(x_{k+1}, \lambda_k) - \sL_\rho(x_k, \lambda_k) \leq - \frac{\beta}{2} \| x_{k+1} - x_k \|_H^2.
%\end{align}
For any $k \geq 1$, we have that
\begin{align}
\notag
& P_{k+1} - P_k = \sL_\rho(x_{k+1}, \lambda_{k+1}) - \sL_\rho(x_k, \lambda_k) + \frac{\beta}{4} \| \Delta x_{k+1}  \|^2 - \frac{\beta}{4} \| \Delta x_k \|^2 \\
\notag
& = \sL_\rho(x_{k+1}, \lambda_{k+1}) - \sL_\rho(x_{k+1}, \lambda_k) + \sL_\rho(x_{k+1}, \lambda_k) - \sL_\rho(x_k, \lambda_k) \\
\notag
& \quad + \frac{\beta}{4} \| \Delta x_{k+1} \|^2 - \frac{\beta}{4} \| \Delta x_k \|^2 \\
\notag
& \rev{= (\lambda_{k+1} - \lambda_k)^T c(x_{k+1}) + \sL_\rho(x_{k+1}, \lambda_k) - \sL_\rho(x_k, \lambda_k)} \\
\notag
& \Rev{ \quad + \frac{\beta}{4} \| \Delta x_{k+1} \|^2 - \frac{\beta}{4} \| \Delta x_k \|^2} \\
\notag
& \Rev{= \frac{1}{\rho} \| \Delta \lambda_{k+1} \|^2 + \sL_\rho(x_{k+1}, \lambda_k) - \sL_\rho(x_k, \lambda_k) + \frac{\beta}{4} \| \Delta x_{k+1} \|^2 - \frac{\beta}{4} \| \Delta x_k \|^2} \\
\notag
& \overset{\eqref{eq:decr}}{\leq} \frac{1}{\rho} \| \Delta \lambda_{k+1}  \|^2 - \frac{\beta}{2} \| \Delta x_{k+1} \|^2 + \frac{\beta}{4} \| \Delta x_{k+1} \|^2 - \frac{\beta}{4} \| \Delta x_k  \|^2 \\
\label{ineq-1: Decr. of Ly. func.}
& = \frac{1}{\rho} \| \Delta \lambda_{k+1}  \|^2 - \frac{\beta}{4} \| \Delta x_{k+1} \|^2 - \frac{\beta}{4} \| \Delta x_k \|^2,
\end{align}
\Rev{where the fourth equality holds because of Step 2 in
  Algorithm~\ref{Alg: Prox-PDA}}. We start with a technical result on
bounding $ \| \Delta \lambda_{k+1} \|^2 = \| \lambda_{k+1} - \lambda_k \|^2 $.
  
\begin{lemma}[Bound for $\| \lambda_{k+1} - \lambda_k  \|^2$]\label{lm: bound for mul. diff.-inexactsubproblem}
Consider Algorithm~\ref{Alg: Prox-PDA} with \eqref{eq:1oi} and
\eqref{eq:decr}. Suppose that for a fixed $k \ge 1$, Assumption~\ref{Ass: weak} holds
  for some set $\sS$ and that $x_k, x_{k+1} \in \sS$. Then,
\begin{equation}
\label{ineq: bound for mul. diff.-inexactsubproblem}
\| \lambda_{k+1} - \lambda_k \|^2 \leq C_1 \| \Delta x_{k+1} \|^2 +
    C_2 \| \Delta x_k \|^2 
 + \frac{16 M_c^2 }{\sigma^4} \| \tr_k \|^2 + \frac{4}{\sigma^2} \| \tr_{k+1} - \tr_k \|^2,
\end{equation}
%% \sjw{Does this hold for $k \ge 1$? We should say so. I suppose
%%   $k=0$ is excluded because $\Delta x_0$ is not defined. Or perhaps
%%   the result holds for $k=0$ and we replace $\Delta x_0=0$ in this
%%   case. Anyway, we should clear this up. \\ Yue: You are
%%   right. This only holds when $k \ge 1$ and I have clarified that.}
where $C_1$ and $C_2$ are defined by
\begin{equation} \label{eq:C1C2}
  C_1 \triangleq   \frac{4}{\sigma^2} \left( L_f +  \frac{ L_c M_f }{\sigma} + \beta \right)^2, \quad
  C_2 \triangleq \frac{4}{\sigma^2} \left( \beta + \frac{2M_c \beta}{\sigma} \right)^2.
  \end{equation}
\end{lemma}
\begin{proof}
The first-order optimality condition \eqref{eq:1oi} for Step 1 implies that for all
$t \geq 0$, we have
\begin{align}
\notag
& \nabla f(x_{t+1}) + \nabla c(x_{t+1}) \lambda_t + \rho \nabla c(x_{t+1}) c(x_{t+1}) + \beta (x_{t+1} - x_t) = \tr_{t+1}. \\
\label{equality1-bound for multiplier (inexact)}
\implies \ &  \nabla f(x_{t+1}) + \nabla c(x_{t+1}) \lambda_{t+1} + \beta (x_{t+1} - x_t) =  \tr_{t+1}.
\end{align}
Likewise, by replacing $t$ with $t-1$, for $t \ge 1$, we obtain
\begin{align} \label{equality2-bound for multiplier (inexact)}
\nabla f(x_t) + \nabla c(x_t) \lambda_t + \beta (x_t - x_{t-1} ) = \tr_t.
\end{align}
By combining \eqref{equality1-bound for multiplier (inexact)} and
\eqref{equality2-bound for multiplier (inexact)} and using the
notation \eqref{eq:DxDl} along with $\Delta \tr_{t+1} \triangleq
\tr_{t+1} - \tr_t$, we have for any $t \ge 1$ that
% \begin{align} \label{equality3-bound for multiplier}
\begin{align*}
\nabla f(x_{t+1}) - \nabla f(x_t) & + \nabla c(x_{t+1}) \Delta \lambda_{t+1}  \\
& + (\nabla c(x_{t+1}) - \nabla c(x_t)) \lambda_t + \beta ( \Delta x_{t+1} - \Delta x_t ) = \Delta \tr_{t+1},
\end{align*}
which by rearrangement gives
\begin{align*}
\notag
- \nabla c(x_{t+1}) \Delta \lambda_{t+1} 
%\label{equality4-bound for multiplier (inexact)}
& = \nabla f(x_{t+1}) -
\nabla f(x_t) + ( \nabla c(x_{t+1}) - \nabla c(x_t) ) \lambda_t  \\
\notag
& \quad  + \beta (\Delta x_{t+1} - \Delta x_t ) - \Delta \tr_{t+1}.
\end{align*}
For the given $k \ge 1$, since $\sigma$ is a lower bound on the smallest singular value of
$\nabla c(x_{k+1})$ by Assumption~\ref{Ass: weak}, we have that
\begin{align}
  \nonumber
  \| \Delta \lambda_{k+1} \| & \le \frac{1}{\sigma}
  \big(
  \| \Rev{\nabla f(x_{k+1})} -\nabla f(x_k) \| + \| \nabla c(x_{k+1}) - \nabla c(x_k) \| \| \lambda_k \|  \\
  \label{eq:us8 (inexact)}
  & \quad +  \beta (\| \Delta x_{k+1}\| + \| \Delta x_{k}\|) +  \| \Delta \tr_{k+1} \|
  \big).
\end{align}
We have from \eqref{equality2-bound for multiplier (inexact)} that
\[
\nabla c(x_k) \lambda_k = -\nabla f(x_k) - \beta (x_k-x_{k-1})  + \tr_k,
\]
so that
\begin{equation} \label{inequality1 - bound for multiplier (inexact)}
  \| \lambda_k \| \le \frac{1}{\sigma} \left( \| \nabla f(x_k) \| +
  \beta \| \Delta x_k\| + \| \tr_k \| \right) \le \frac{1}{\sigma} \left(
  M_f + \beta \| \Delta x_k\|  + \| \tr_k \| \right).
  \end{equation}
We also have from Assumption~\ref{Ass: weak} that
\begin{align}\label{inequality2 - bound for multiplier (inexact)}
  \| \nabla c(x_{k+1}) - \nabla c(x_k) \| \leq L_c \| x_{k+1} - x_k \|, \quad
   \| \nabla c(x_{k+1}) - \nabla c(x_k) \| \leq 2 M_c.
\end{align}
By substituting Assumption~\ref{Ass: weak}(i),
\eqref{inequality1 - bound for multiplier (inexact)}, and \eqref{inequality2 -
  bound for multiplier (inexact)} into \eqref{eq:us8 (inexact)}, we obtain the following for the given $k \ge 1$.
\begin{align*}
& \| \Delta \lambda_{k+1} \| \\
  & \leq \frac{1}{\sigma} \bigg( L_f \| \Delta x_{k+1} \| + \beta \| \Delta x_{k+1} \| + \beta \| \Delta x_k \|  \\
 &  \quad + \| \nabla c(x_{k+1}) - \nabla c(x_k) \|_2 \left( \frac{1}{\sigma} M_f + \frac{\beta}{\sigma} \| \Delta x_k \|  + \frac{1}{\sigma} \| \tr_k \| \right)  + \| \Delta \tr_{k+1} \|  \bigg) \\
  & \le \frac{1}{\sigma}\bigg( L_f \| \Delta x_{k+1}\| + \beta \| \Delta x_{k+1} \| + \beta \| \Delta x_k \| + \frac{L_c M_f }{\sigma} \| \Delta x_{k+1} \|  + \frac{2M_c \beta}{\sigma} \| \Delta x_k \| \\
&  \quad + \frac{2M_c}{\sigma} \| \tr_k \| + \| \Delta \tr_{k+1} \| \bigg) \\
    & \le \frac{1}{\sigma} \left( L_f +  \frac{L_c M_f}{\sigma} + \beta \right) \| \Delta x_{k+1} \| +
    \frac{1}{\sigma} \left( \beta + \frac{2M_c \beta}{\sigma} \right) \| \Delta x_k \| \\
  & \quad + \frac{2M_c}{\sigma^2} \| \tr_k \| + \frac{1}{\sigma} \| \Delta \tr_{k+1} \|.
\end{align*}
By using the bound $(a + b + c +d)^2 \le 4(a^2 + b^2 + c^2 + d^2)$ for
positive scalars $a$, $b$, $c$, $d$, and using the definition
\eqref{eq:C1C2}, we obtain the result. \qed
\end{proof}

For the rest of this section, we use the following definitions for
$c_1$ and $c_2$:
\begin{align}
\label{def-c1c2-inexact}
& c_1 \triangleq \frac{\beta}{4} - \frac{C_1}{\rho}, \quad
c_2 \triangleq \frac{\beta}{4} - 
\frac{C_2}{\rho},
\end{align}
where $C_1$ and $C_2$ are defined in \eqref{eq:C1C2}. Next we show
that sequences $\{ x_k \}$ and $\{ \lambda_k \}$ are bounded and $\{
P_k \}_{k \geq 1} $ satisfies certain properties under
Assumption~\ref{Ass: Compact sublevel set} - \ref{Ass: f.ub}, for
suitable choices of the algorithmic parameters.

\begin{lemma} \label{lm: xk.bd.} 
\Rev{
Consider Algorithm~\ref{Alg: Prox-PDA} with conditions \eqref{eq:1oi}
and \eqref{eq:decr}. Choose $\{ \tr_k \}_{k \ge 1}$ such that
$\sum_{k=1}^\infty \| \tr_k \|^2 \le R < +\infty$ and $\| \tr_k \| \le
1$, for all $k \ge 1$. Let $\{ P_k \}_{k \geq 1}$ be defined as in
\eqref{Def: P-func}. Suppose that Assumptions~\ref{Ass: Compact
  sublevel set} and \ref{Ass: f.ub} hold and define
%% \sjw{Oh, I just noticed that $c(x_0)=0$. That's a pretty strong
%%   assumption! In general, we'd need a Phase I to get this. We
%%   should check whether we can prove a similar result without this
%%   assumption. (We discussed this later on a call.)\\ Yue: It's
%%   relaxed finally.}
\begin{equation} \label{eq:alpha_l0-inexact}
  \hat \alpha \triangleq 7 \bar U + 7 C_0 - 6 \bar L + 13 \| \lambda_0 \|^2 + 2,
%\;\; \mbox{where} \;\;
%  \bar L \triangleq \inf_{x \in \bR^n} \left\{ f(x)
%  + \frac{\rho_0}{2} \| c(x) \|^2 \right\},
\end{equation}
where $C_0 > 0$ is any fixed constant.
Suppose that Assumption~\ref{Ass: weak} holds with $\sS = S_{\hat
  \alpha}^0$. Choose $\rho$ and $\beta$ such that
\begin{equation} \label{ineq: rho.bd.}
\rho \geq \max \left\{ \frac{ ( M_f + \beta D_S + 1 )^2 }{2\sigma^2} +
\rho_0, \frac{16 (M_c^2 + \sigma^2) R }{ \sigma^4 }, 3 \rho_0, 1 \right\},
\end{equation}
where 
\begin{align} \label{def: DS}
D_S \triangleq \max\{ \| x - y \| \mid x, y \in S_{\hat \alpha}^0
\},
\end{align}
and that $c_1$ and $c_2$ defined in \eqref{def-c1c2-inexact}
are both positive. Suppose that $x_0$ satisfies $\| c(x_0) \|^2
  \le \min \{ C_0/\rho , 1 \}$, where $C_0$ is the constant appearing
  in \eqref{eq:alpha_l0-inexact}. Then
\begin{equation} \label{eq:sj7}
\{ x_k \}_{k
  \geq 0} \subseteq S_{\hat \alpha}^0 \quad \mbox{and} \quad  \|\lambda_k \| \leq \frac{M_f
+ \beta D_S + 1}{\sigma}, \;\; \mbox{for all $k \ge 1$.}
\end{equation}
Furthermore, \eqref{ineq: bound for
  mul. diff.-inexactsubproblem} and the following inequality hold for
any $k \geq 1$,
\begin{equation} \label{ineq: Decr. of Ly. func.-inexact}
P_{k+1} - P_k \le - c_1 \| \Delta x_{k+1} \|^2 - c_2 \| \Delta x_k \|^2 + \frac{16M_c^2}{ \rho \sigma^4} \| \tr_k \|^2 + \frac{4}{\rho \sigma^2} \| \tr_{k+1} - \tr_k \|^2.
\end{equation}}
\end{lemma}
\begin{proof}
Note that Assumption~\ref{Ass: f.ub} implies that
\begin{align} \label{ineq: f.ub.}
f(x_0) \le \bar U,
\end{align}
since $\| c(x_0) \| \le 1$. Therefore,
\begin{align} \label{ineq: AL0.bd.}
\notag
\Lscr_\rho(x_0, \lambda_0) & = f(x_0) + \lambda_0^T c(x_0) + \frac{\rho}{2} \| c(x_0) \|^2 \\
\notag
& \le f(x_0) + \frac{ \| \lambda_0 \|^2 }{ 2 \rho } + \frac{\rho}{2} \| c(x_0) \|^2 + \frac{\rho}{2} \| c(x_0) \|^2 \\
& \overset{ \eqref{ineq: f.ub.} }{\le} \bar U + \frac{1}{2 \rho} \| \lambda_0 \|^2 + C_0.
\end{align}
and
\begin{align}
\notag
\bar U + C_0 - \bar L & \overset{ \eqref{ineq: AL0.bd.} }{\ge} f(x_0) + \lambda_0^T c(x_0) + \frac{\rho}{2} \| c(x_0) \|^2 - \frac{ \| \lambda_0 \|^2 }{ 2\rho } - \bar L \\
\notag
& \overset{ (\rho \ge 3 \rho_0) }{\ge} f(x_0) + \frac{\rho_0}{2} \| c(x_0) \|^2 - \bar L  + \lambda_0^T c(x_0) + \frac{\rho}{3} \| c(x_0) \|^2 - \frac{ \| \lambda_0 \|^2 }{ 2\rho } \\
\notag
& \ge 0 + \frac{\rho}{3} \left\| c(x_0) + \frac{3\lambda_0}{ 2\rho } \right\|^2 - \frac{3}{4 \rho} \| \lambda_0 \|^2 - \frac{ \| \lambda_0 \|^2 }{ 2\rho } \\
\label{ineq: nn0}
& \overset{ ( \rho \ge 1 ) }{\ge} - \frac{5}{4} \| \lambda_0 \|^2
\end{align}
We prove the theorem by induction. We show that the following
bounds hold for all $i \geq 1$:
\begin{subequations} \label{Hyp-inexact}
  \begin{align}
    \label{Hyp-inexact.1}
    x_i & \in S_{\hat \alpha}^0, \\
    \label{Hyp-inexact.2}
    \|  \lambda_i \|^2 & \leq \frac{ (M_f + \beta D_S + 1 )^2 }{ \sigma^2 } \leq 2(\rho - \rho_0), \\
    \label{Hyp-inexact.3}
P_i & \leq 7 \bar U + 7C_0 - 6 \bar L + 13\| \lambda_0 \|^2 + \frac{16M_c^2}{\rho \sigma^4 } \sum_{t=1}^{i-1} \| \tr_t \|^2 +  \frac{4}{ \rho \sigma^2} \sum_{t=1}^{i-1} \| \tr_{t+1} - \tr_t \|^2.
\end{align}
\end{subequations}
We verify first that \eqref{Hyp-inexact} holds when $i = 1$. From
\eqref{eq:decr} we have
\begin{align} \label{ineq: nn1}
\notag
f(x_1) + \lambda_0^T c(x_1) & + \frac{\rho}{2} \| c(x_1) \|^2 + \frac{\beta}{2} \| x_1 - x_0 \|^2 \\
& \le
f(x_0) + \lambda_0^T c(x_0) + \frac{\rho}{2} \| c(x_0) \|^2 \overset{ \eqref{ineq: AL0.bd.} }{\le} \bar U + \frac{\| \lambda_0 \|^2}{2\rho} + C_0, 
\end{align}
so that for $i = 0$ and $1$, we have
\begin{align*}
f(x_i) + \frac{\rho}{6} \| c(x_i) \|^2 & \overset{\eqref{ineq: AL0.bd.}, \eqref{ineq: nn1}}{\le} \bar U + \frac{\| \lambda_0 \|^2}{2\rho} + C_0 - \lambda_0^T c(x_i) - \frac{\rho}{3} \| c(x_i) \|^2 \\
& = \bar U + \frac{\| \lambda_0 \|^2}{2\rho} + C_0 - \frac{\rho}{3} \left\| c(x_i) + \frac{3\lambda_0}{2\rho} \right\|^2 + \frac{ 3 \| \lambda_0 \|^2}{4\rho} \\
\overset{ (\rho \geq 3 \rho_0) }{\implies} f(x_i) + \frac{\rho_0}{2} \| c(x_i) \|^2 & \le \bar U + \frac{5 \| \lambda_0 \|^2}{4\rho} + C_0 \\
& \hspace{-1in} \overset{ ( \eqref{ineq: nn0} , \rho \geq 1) }{\leq} \bar U + C_0 + \frac{5}{4} \| \lambda_0 \|^2 + 6 \left( \bar U + C_0 - \bar L + \frac{5}{4} \| \lambda_0 \|^2  \right) \\
& \hspace{-.8in} \le 7 \bar U + 7 C_0 - 6 \bar L + \frac{35}{4} \| \lambda_0 \|^2 \overset{\eqref{eq:alpha_l0-inexact}}{<} \hat{\alpha}.
\end{align*}
 Thus, $x_0, x_1 \in S_{\hat \alpha}^0$, verifying that \eqref{Hyp-inexact.1}
 holds for $i=1$.

 Approximate first-order optimality \eqref{eq:1oi} indicates that
\begin{align*}
\nabla f(x_1) + \nabla c(x_1) \lambda_1 + \beta( x_1 - x_0) = \tr_1.
\end{align*}
Since $x_0, x_1 \in S_{\hat \alpha}^0$, we have by Assumption~\ref{Ass: weak} and \eqref{def: DS} that
\begin{align*}
\sigma \| \lambda_1 \| & \leq \| \nabla c(x_1) \lambda_1 \| = \| \nabla f(x_1) + \beta (x_1 - x_0) - \tr_1 \| \leq M_f + \beta D_S + 1.\\ 
\implies  \; \| \lambda_1 \|^2 & \leq \frac{(M_f + \beta D_S + 1 )^2}{\sigma^2} \overset{\eqref{ineq: rho.bd.}}{\leq} 2(\rho - \rho_0).
\end{align*}
%% where the last inequality follows from the definition of $\rho$ in
%% \eqref{ineq: rho.bd.}.
Thus, \eqref{Hyp-inexact.2} holds for $i=1$.

Next, we verify \eqref{Hyp-inexact.3} when $i=1$. Note that
\begin{align}
\notag
P_1 & = \Lscr_\rho(x_1, \lambda_1) + \frac{\beta}{4} \| x_1 - x_0 \|^2 \\
\notag
& = \Lscr_\rho(x_1, \lambda_1) - \Lscr_\rho(x_1, \lambda_0) + \Lscr_\rho(x_1, \lambda_0) - \Lscr_\rho(x_0, \lambda_0) + \Lscr_\rho(x_0, \lambda_0) \\
\notag
& \quad + \frac{\beta}{4} \| x_1 - x_0 \|^2 \\
\notag
& \overset{\eqref{eq:decr}}{\leq} \frac{1}{\rho} \| \lambda_1 - \lambda_0 \|^2 - \frac{\beta}{2} \| x_1 - x_0 \|^2 + \Lscr_\rho(x_0, \lambda_0) + \frac{\beta}{4} \| x_1 - x_0 \|^2 \\
\notag
& = \rho \| c(x_1) \|^2 -  \frac{\beta}{4}  \| x_1 - x_0 \|^2 + \Lscr_\rho(x_0,\lambda_0) \\
%% \notag
%% & \rev{= \rho \| c(x_1) \|^2 - \left( c_1 + \frac{2C_1}{\rho} \right) \| x_1 - x_0 \|^2 + f(x_0) } \\
\notag
& \overset{\eqref{ineq: AL0.bd.}}{\le} \rho \| c(x_1) \|^2 + \bar U + \frac{1}{2 \rho} \| \lambda_0 \|^2 + C_0, \\
\label{ineq-1: complexity-exact case}
& \overset{ (\rho \ge 1) }{\le} \rho \| c(x_1) \|^2 + \bar U + \frac12 \| \lambda_0 \|^2 + C_0,
\end{align}
In addition, \eqref{ineq: nn1} indicates that 
\begin{align}
\notag
& \frac{\rho}{6} \| c(x_1) \|^2 \\
\notag
& \leq \bar U + \frac{1}{2\rho} \| \lambda_0 \|^2  + C_0 - \lambda_0^T c(x_1) - \frac{\rho}{6} \| c(x_1) \|^2 - f(x_1) - \frac{\rho}{6} \| c(x_1) \|^2 \\
\notag
& = \bar U + \frac{1}{2\rho} \| \lambda_0 \|^2 + C_0  - \frac{\rho}{6} \| c(x_1) + 3 \lambda_0 / \rho \|^2 + \frac{3 \| \lambda_0 \|^2}{ 2 \rho }  - f(x_1) - \frac{\rho}{6} \| c(x_1) \|^2 \\
\notag
& \overset{ (\rho \geq 3 \rho_0) }{\leq} \bar U + \frac{1}{2\rho} \| \lambda_0 \|^2 + C_0 + \frac{3 \| \lambda_0 \|^2 }{2 \rho} - f(x_1) - \frac{\rho_0}{2} \| c(x_1) \|^2 \\
\notag
& \leq \bar U + \frac{2}{\rho} \| \lambda_0 \|^2 + C_0 - \bar L \overset{ ( \rho \ge 1 ) }{\leq} \bar U + 2 \| \lambda_0 \|^2 + C_0 - \bar L.
\end{align}
By substituting this bound into \eqref{ineq-1: complexity-exact
  case}, we have that
\begin{equation}\label{Bound on P1}
P_1 \le \bar U + \frac{\| \lambda_0 \|^2}{2} + C_0 + \rho \|c(x_1) \|^2 \le 7 \bar U + 7C_0 - 6 \bar L + 13 \| \lambda_0 \|^2,
\end{equation}
so \eqref{Hyp-inexact.3} holds for $i=1$ also.

We now take the inductive step, supposing that \eqref{Hyp-inexact}
holds when $i = k \geq 1$, and proving that these three conditions
continue to hold for $i = k+1$.  By inequality \eqref{eq:decr}, we
have
\begin{align*}
& \quad \sL_\rho(x_{k+1},\lambda_k) \le \sL_\rho(x_k,\lambda_k) \le P_k  \\
& \implies f(x_{k+1}) + \frac{\rho}{2} \| c(x_{k+1}) \|^2 + \lambda_k^T c(x_{k+1}) \leq P_k \\
& \implies f(x_{k+1}) + \frac{\rho}{2} \| c(x_{k+1}) \|^2 - \frac{ \| \lambda_k \|^2 }{ 2(\rho - \rho_0) } - \frac{ ( \rho - \rho_0 ) \| c(x_{k+1}) \|^2}{2} \leq P_k \\
& \implies f(x_{k+1}) + \frac{\rho_0}{2} \| c(x_{k+1}) \|^2 \leq P_k + \frac{ \| \lambda_k \|^2 }{2(\rho - \rho_0)} \overset{ \eqref{Hyp-inexact.2} }{\leq} P_k + 1 \\
& \quad \overset{\eqref{Hyp-inexact.3}}{\le} 
 7 \bar U + 7C_0 - 6 \bar L + 13 \| \lambda_0 \|^2 + \frac{16M_c^2}{\rho \sigma^4 } \sum_{t=1}^{k-1} \| \tr_t \|^2 +  \frac{4}{ \rho \sigma^2} \sum_{t=1}^{k-1} \| \tr_{t+1} - \tr_t \|^2 + 1 \\
  & \quad \le 7 \bar U + 7C_0 - 6 \bar L + 13 \| \lambda_0 \|^2 + \frac{16M_c^2}{\rho \sigma^4 } \sum_{t=1}^{k-1} \| \tr_t \|^2 +  \frac{8}{ \rho \sigma^2} \sum_{t=1}^{k-1} ( \| \tr_{t+1} \|^2 + \| \tr_t \|^2 ) +1 \\
& \quad \le 7 \bar U + 7C_0 - 6 \bar L + 13 \| \lambda_0 \|^2 + \frac{16M_c^2}{\rho \sigma^4 } \sum_{t=1}^{\infty} \| \tr_t \|^2 +  \frac{16}{ \rho \sigma^2} \sum_{t=1}^{\infty} \| \tr_t \|^2  + 1 \\
& \quad \le 7 \bar U + 7C_0 - 6 \bar L + 13 \| \lambda_0 \|^2 + \frac{ 16 (M_c^2 + \sigma^2 ) R }{\rho \sigma^4 } + 1 \\
& \quad \overset{ \eqref{ineq: rho.bd.} }{\le} 7 \bar U + 7C_0 - 6 \bar L + 13 \| \lambda_0 \|^2 + 2 = \hat{\alpha},
\end{align*}
where the inequality on the third line holds because of $ -\frac{r}{2} \|
a \|^2 - \frac{1}{2r} \| b \|^2 \leq a^Tb$, for any $r > 0$, $a,b \in
\bR^m$. Therefore, $ x_{k+1} \in S_{\hat \alpha}^0 $, so we have
proved \eqref{Hyp-inexact.1}.

By approximate first-order optimality \eqref{eq:1oi} and
the hypothesis $x_k \in S_{\hat \alpha}^0$, the argument to establish
that $\| \lambda_{k+1} \|^2 \leq \frac{ (M_f + \beta D_S + 1 )^2
}{ \sigma^2 } \leq 2 (\rho - \rho_0)$ is the same as for the case of $
i = 1 $, so \eqref{Hyp-inexact.2} holds for $i=k+1$.

Since $x_k$ and $x_{k+1}$ both belong to $S_{\hat \alpha}^0$,
Lemma~\ref{lm: bound for mul. diff.-inexactsubproblem} indicates that
\eqref{ineq: bound for mul. diff.-inexactsubproblem} holds. By
  combining \eqref{ineq: bound for mul. diff.-inexactsubproblem} with
  \eqref{ineq-1: Decr. of Ly. func.}, we obtain \eqref{ineq: Decr. of
    Ly. func.-inexact}. Therefore,
\begin{align*}
P_{k+1} & \overset{ \eqref{ineq: Decr. of Ly. func.-inexact} }{\le} P_k + \frac{16 M_c^2}{ \rho \sigma^4} \| \tr_k \|^2 + 
 \frac{4}{ \rho \sigma^2} \| \tr_{k+1} - \tr_k \|^2 \\
& \overset{ \eqref{Hyp-inexact} }{\le} 7 \bar U + 7C_0 - 6 \bar L + 13 \| \lambda_0 \|^2 + \frac{16M_c^2}{\rho \sigma^4 } \sum_{t=1}^{k} \| \tr_t \|^2 +  \frac{4}{ \rho \sigma^2} \sum_{t=1}^{k} \| \tr_{t+1} - \tr_t \|^2.
\end{align*}
Thus we have established \eqref{Hyp-inexact.3} for $i = k+1$. Note
that \eqref{ineq: bound for mul. diff.-inexactsubproblem} and
\eqref{ineq: Decr. of Ly. func.-inexact} hold for all $k \geq 1$, so
we have completed the proof. \qed
\end{proof}

\paragraph{First-order  complexity.}
With the properties of $\{P_k\}_{k \geq 1}$ established to this point,
we can analyze the complexity of obtaining an $\epsilon$-1o
solution. For any given $\epsilon > 0$, we define
  two quantities which will be referred to repeatedly in subsequent
  sections:
  \begin{subequations}
    \label{def:bothTeps}
\begin{align}
\label{def: Teps}
T_\epsilon & \triangleq \inf\{ t \ge 1 \mid \| \nabla_x \Lscr_0(x_t, \lambda_t) \| \le \epsilon, \| c(x_t) \| \le \epsilon \}. \\
\label{def: hTeps}
\hat{T}_\epsilon  & \triangleq \inf\{ t \ge 1 \mid \mbox{$x_t$ is an $\epsilon$-1o solution of \eqref{eqcons-opt}} \}.
\end{align}
\end{subequations}
Note that $\hat{T}_\epsilon$ is independent of the Proximal AL
method. Meanwhile, by the definition of $\Lscr_0(x, \lambda)$, we know
that $x_{T_\epsilon}$ is an $\epsilon$-1o solution and
$\lambda_{T_\epsilon}$ is the associated multiplier, indicating that
$\hat{T}_\epsilon \le T_\epsilon$.  The definition of $T_\epsilon$
also suggests the following stopping criterion for Algorithm~\ref{Alg:
  Prox-PDA}:
\begin{equation} \label{stopcrit}
\mbox{If $\| \nabla_x \Lscr_0(x_t, \lambda_t) \| \le \epsilon$ and  $\| c(x_t) \| \le \epsilon$ then STOP.}
\end{equation}
Under this criterion, Algorithm~\ref{Alg: Prox-PDA} will stop at
iteration \Rev{$T_\epsilon - 1$} and output $x_{T_\epsilon}$ as an $\epsilon$-1o solution.

Part (i) of the following result shows subsequential convergence of
the generated sequence to the first-order optimal point. Part (ii)
describes the speed of such convergence by obtaining an estimate of
$T_{\epsilon}$ in terms of $\epsilon$. In this result, we make a specific choice $\beta = \epsilon^{\eta}/2$ for the proximality
  parameter. We could choose $\beta$ to be any fixed multiple of this
  value (the multiple not depending on $\epsilon$) and obtain a
  similar result with only trivial changes to the analysis.
\begin{theorem}[First-order complexity] \label{thm: first-order complexity - inexactsubproblem} 
\Rev{
Consider Algorithm~\ref{Alg: Prox-PDA} with conditions \eqref{eq:1oi}
and \eqref{eq:decr}, and let $\{P_k\}_{k \geq 1}$ be defined as in
\eqref{Def: P-func}. Suppose that Assumption~\ref{Ass: Compact
    sublevel set}, \ref{Ass: f.ub} and Assumption~\ref{Ass: weak} hold with $\sS =
  S_{\hat \alpha}^0$ (with $\hat \alpha$ defined in
  \eqref{eq:alpha_l0-inexact}), and that $\epsilon \in (0,1]$ and
  $\eta \in [0,2]$ are given. Suppose that the residual sequence $\{
  \tr_k \}_{k \geq 1}$ is chosen such that $ \sum_{k=1}^\infty \|
  \tr_k \|^2 \leq R \in [1, \infty) $ and $\| \tr_k \| \leq
    \epsilon/2$ for all $k \ge 1$.
    %%    \sjw{I didn't notice this restriction on $\| \tr_k \|$
    %%    before - it is pretty strong!\\ Yue: Maybe. Yet it
    %%    simplifies the proof and may not worsen the total iteration
    %%    complexity.  }
    Define $\beta = \epsilon^\eta/2$  and 
\begin{equation} \label{Parameter Ass.-inexact}
\rho \ge \max \left\{ \frac{16 \max\{ C_1, C_2 \}}{ \epsilon^\eta }, \frac{(M_f + \beta D_S + 1)^2}{2 \sigma^2} + \rho_0, \frac{16(M_c^2 + \sigma^2) R }{ \sigma^4 }, 3 \rho_0, 1 \right\},
\end{equation}
where $C_1$ and $C_2$ are defined as in \eqref{eq:C1C2}, and $D_S$ is
the diameter of $S_{\hat \alpha}^0$, as defined in \eqref{def:
  DS}. Suppose that $x_0$ satisfies $ \| c(x_0) \|^2 \le \min\{
  C_0 /\rho , 1 \} $, where $C_0$ is the constant appearing in
  \eqref{eq:alpha_l0-inexact}. Then we have the following.
  \begin{itemize}
    \item[(i)] A subsequence of $\{ ( x_k,
  \lambda_k ) \}_{k \ge 1}$ generated by Algorithm~\ref{Alg: Prox-PDA}
  converges to a point $(x^*, \lambda^*)$ satisfying first-order
  optimality conditions for \eqref{eqcons-opt}, namely,
\[
\nabla f(x^*) + \nabla c(x^*) \lambda^* =0, \quad c(x^*) = 0.
\]
\item[(ii)] For $T_\epsilon$ and $\hat{T}_\epsilon$ defined in
\eqref{def:bothTeps}, we have $\hat{T}_\epsilon \le T_\epsilon = \sO(
1/\epsilon^{2-\eta} )$. In particular, if $\eta = 2$, then
$\hat{T}_{\epsilon} = \sO(1)$.
\end{itemize}}
\end{theorem}
\begin{proof}
We first prove (i). Checking the positivity of $c_1$ and
$c_2$, given the parameter assignments, we have
\begin{equation} \label{eq:c1c2.bd}
 c_1 = \frac{\beta}{4} - \frac{C_1}{\rho} \overset{\eqref{Parameter Ass.-inexact}}{\geq} \frac{\epsilon^\eta}{8} - \frac{ \epsilon^\eta}{16}  > 0, \quad c_2 = \frac{\beta}{4} - \frac{C_2}{\rho} \overset{\eqref{Parameter Ass.-inexact}}{\geq} \frac{ \epsilon^\eta}{16} > 0.
\end{equation}
%% By \eqref{ineq: Decr. of Ly. func.-inexact} from Lemma~\ref{lm: xk.bd.}, we have
%% for any $k \geq 1$ that
%% \begin{align}
%% \notag
%% & P_{k+1} - P_k \\
%% \notag
%% & \le - \hat{c}_1 \| \Delta x_{k+1}  \|^2 - \hat{c}_2 \| \Delta x_k  \|^2 + \frac{16 M_c^2}{ \rho \sigma^4} \| \tr_k \|^2 + \frac{4}{ \rho \sigma^2} \| \tr_{k+1} - \tr_k \|^2 \\
%% \label{ineq: nn6.5}
%% & \leq - \hat{c}_1 \| \Delta x_{k+1}  \|^2 - \hat{c}_2 \| \Delta x_k  \|^2 + \frac{16M_c^2 + 8 \sigma^2}{\rho \sigma^4 } \| \tr_k \|^2 +  \frac{8}{ \rho \sigma^2} \| \tr_{k+1} \|^2.
%% \end{align}ve
\begin{align}
\notag
P_k & \ge f(x_k) + \frac{\rho}{2} \| c(x_k) \|^2 + \lambda_k^T c(x_k) \\
\notag
& \ge f(x_k) + \frac{\rho}{2} \| c(x_k) \|^2 - \frac{ \| \lambda_k \|^2 }{ 2(\rho - \rho_0) } - \frac{ (\rho-\rho_0) \| c(x_k) \|^2 }{2} \\
\notag
& = f(x_k) + \frac{\rho_0}{2} \| c(x_k) \|^2 - \frac{ \| \lambda_k \|^2 }{ 2(\rho - \rho_0) } \\
\notag
& \overset{ (\textrm{Lemma}~\ref{lm: xk.bd.}) }{\ge} f(x_k) + \frac{\rho_0}{2} \| c(x_k) \|^2 - \frac{ (M_f + \beta D_S + 1)^2 }{ 2 \sigma^2 (\rho - \rho_0) } \\
\label{ineq: Pk.low.bd.}
& \overset{ \eqref{eq:defL}, \eqref{Parameter Ass.-inexact} }{\ge} \bar L - 1.
\end{align}
Therefore, using \eqref{ineq: Decr. of
  Ly. func.-inexact} from Lemma~\ref{lm: xk.bd.}, we have the following for any  $k \geq 1$:
\begin{align}
\notag
& \sum_{i=1}^k \left[ c_1 \| \Delta x_{i+1} \|^2 + c_2 \| \Delta x_i \|^2 \right] \\
\notag
& \leq P_1 - P_{k+1} + \frac{16M_c^2}{\rho \sigma^4 } \sum_{i=1}^k \| \tr_i \|^2 + \frac{4}{ \rho \sigma^2} \sum_{i=1}^k \| \tr_{i+1} -\tr_i \|^2 \\
\notag
& \leq P_1 - P_{k+1} + \frac{16M_c^2}{\rho \sigma^4 } \sum_{i=1}^k \| \tr_i \|^2 + \frac{8}{ \rho \sigma^2} \sum_{i=1}^k  (\| \tr_{i+1}\|^2 + \|\tr_i \|^2)\\
\notag
& \leq P_1 - P_{k+1} + \frac{16 (M_c^2+ \sigma^2) }{\rho \sigma^4 } \sum_{i=1}^\infty \| \tr_i \|^2 \\
\label{ineq: nn7}
& \leq P_1 - P_{k+1} + \frac{16 (M_c^2 + \sigma^2) R }{\rho \sigma^4 } \\
\notag
& \overset{ \eqref{ineq: Pk.low.bd.} }{\leq} P_1 - \left( \bar L - 1 \right) +  \frac{16 (M_c^2 + \sigma^2) R }{\rho \sigma^4 } \\
\notag
& \overset{\eqref{Parameter Ass.-inexact}}{\leq} P_1 - \bar L + 2 \\
\label{ineq2: summability}
& \overset{ \eqref{Bound on P1} }{\le} 7 \bar U + 7C_0  - 6 \bar L + 13 \| \lambda_0 \|^2 - \bar L + 2 = \hat \alpha - \bar L.
\end{align}
 Because of \eqref{eq:sj7} in Lemma~\ref{lm: xk.bd.}
and compactness of $S_{\hat \alpha}^0$, the sequence $\{ (x_k,
\lambda_k) \}_{k \geq 1}$ is bounded, so there exists a convergent
subsequence $\{ (x_k, \lambda_k) \}_{k \in \sK}$ with limit
$(x^*,\lambda^*)$.  Since \eqref{ineq2: summability} holds for any $k
\ge 1$ and $c_1 > 0$, $c_2 > 0$, we have that
$\lim\limits_{k \rightarrow \infty} \| \Delta x_k \| = 0$. Moreover,
finiteness of $\sum_{k=1}^{\infty} \| \tr_k \|^2$ implies that
$\lim\limits_{k \rightarrow \infty} \| \tr_k \| = 0$. Therefore, we have
\begin{align*}
& \nabla f(x^*)
+ \nabla c(x^*) \lambda^* = \lim\limits_{k \in \sK} ( \nabla f(x_k) +
\nabla c(x_k) \lambda_k ) \\
& = \lim\limits_{k \in \sK} ( \nabla f(x_k) +
\nabla c(x_k) ( \lambda_{k-1} + \rho c(x_k)  ) ) = \lim\limits_{k \in \sK}  \nabla_x \Lscr_{\rho}( x_k, \lambda_{k-1}) \\
& \overset{\eqref{eq:1oi}}{=} \lim\limits_{k \in
  \sK} ( - \beta \Delta x_k + \tr_k ) = 0.
\end{align*}
Since \eqref{ineq: bound for mul. diff.-inexactsubproblem} holds for any $k \ge 1$ by Lemma~\ref{lm: xk.bd.}, we have 
\begin{align*}
& \| c(x^*) \|^2 =  \lim\limits_{k \in \sK} \, \|c(x_k) \|^2 = \lim\limits_{k \in \sK} \, \|
\lambda_k - \lambda_{k-1} \|^2/ \rho^2 \\
& \overset{\eqref{ineq: bound for mul. diff.-inexactsubproblem}}{\le } \lim_{k \in \sK} \frac{C_1}{\rho^2} \| \Delta x_k \|^2 +
    \frac{C_2}{\rho^2} \| \Delta x_{k-1} \|^2 
 + \frac{16 M_c^2 }{\rho^2 \sigma^4} \| \tr_{k-1} \|^2 + \frac{4}{ \rho^2 \sigma^2} \| \tr_k - \tr_{k-1} \|^2 =0, 
\end{align*}
%% \sjw{Shouldn't you say $\| c(x^*) \|^2 = \lim_{k \in \sK} \| c(x_k)
%%   \|^2$ and shift all the indices $k$ in this argument back by 1?
%%   It would still work I think. \\ Yue: Yeah. That will be less
%%   confusing. Corrected.}
completing the proof of (i).

We now prove (ii). Define
\begin{subequations}
\begin{align}  
\label{eq:C1o} 
C_1^o & \triangleq \frac{4}{\sigma^2} \left( L_f + \frac{L_c M_f}{\sigma} \right)^2 \le C_1, \\
\label{def: hatDta}
\hat \Delta & \triangleq ( \hat \alpha - \bar L ) \max \left\{ 16, 1/(8 C_1^o)  \right\}.
\end{align}
\end{subequations}
We want to show that $T_\epsilon \leq \lceil \hat{\Delta}/ \epsilon^{2 -
  \eta} \rceil + 1$.  
Let $K \triangleq \lceil \hat \Delta / \epsilon^{2 - \eta} \rceil$, and
note that since \eqref{ineq2: summability} holds for $k=K$, we have that
there exists some  $k^* \in \{1,2,\dotsc,K\}$ such that
\begin{align} \label{ineq-1: complexity-inexact}
c_1 \| x_{k^*+1} - x_{k^*} \|^2 + c_2 \| x_{k^*} - x_{k^*-1} \|^2 \leq (\hat \alpha - \bar L) /K.
\end{align}
Thus, we have
\begin{align*}
& \| \nabla \Lscr_0(x_{k^*+1}, \lambda_{k^*+1}) \| = \| \nabla \Lscr_{\rho}(x_{k^*+1}, \lambda_{k^*} ) \| \overset{\eqref{eq:1oi}}{\leq} \beta \| x_{k^*+1} - x_{k^*} \| + \| \tr_{k^*+1} \| \\
& \overset{\eqref{ineq-1: complexity-inexact}}{\leq} \beta \sqrt{ \frac{ (\hat \alpha - \bar L) /c_1 }{K} } + \frac{\epsilon}{2} \overset{\eqref{eq:c1c2.bd}}{\leq} \frac{ \epsilon^{\eta} }{2} \sqrt{ \frac{ (\hat \alpha - \bar L) / (  \epsilon^{\eta} / 16) }{K} } + \frac{\epsilon}{2} \\
& \leq \frac{  \epsilon^\eta}{2} \sqrt{ \frac{ 16 (\hat \alpha - \bar L) / ( \epsilon^\eta ) }{ \hat \Delta \epsilon^{\eta - 2} } } + \frac{\epsilon}{2} \overset{\eqref{def: hatDta}}{\leq} \frac{ \epsilon^\eta}{2} \sqrt{ \frac{ 16 (\hat \alpha - \bar L) }{ 16 (\hat \alpha - \bar L) \epsilon^{2\eta - 2}  } } + \frac{ \epsilon}{2} = \epsilon.
\end{align*}
For the constraint norm, we have
\begin{align*}
& \| c(x_{k^*+1}) \|^2 = \| \Delta \lambda_{k^*+1}  \|^2 / \rho^2 \\
& \overset{\eqref{ineq: bound for mul. diff.-inexactsubproblem}}{\leq} \frac{C_1}{\rho^2} \| \Delta x_{k^*+1}  \|^2 + \frac{C_2}{\rho^2} \| \Delta x_{k^*} \|^2 
 + \frac{16 M_c^2}{\rho^2 \sigma^4} \| \tr_{k^*} \|^2 + \frac{4}{\rho^2 \sigma^2} \| \tr_{k^*+1} - \tr_{k^*} \|^2 \\
& \leq \frac{C_1}{\rho^2} \| \Delta x_{k^*+1} \|^2 + \frac{C_2}{\rho^2} \| \Delta x_{k^*}  \|^2  + \frac{16 M_c^2}{\rho^2 \sigma^4} \| \tr_{k^*} \|^2 + \frac{8}{\rho^2 \sigma^2} (\| \tr_{k^*} \|^2 + \| \tr_{k^*+1} \|^2) \\
& \leq \frac{C_1}{\rho^2} \| \Delta x_{k^*+1}  \|^2 + \frac{C_2}{\rho^2} \| \Delta x_{k^*}  \|^2 + \frac{16 (M_c^2 + \sigma^2) }{\rho^2 \sigma^4 } \cdot \frac{\epsilon^2}{4}  \\
& \leq \frac{1}{\rho^2}  \max\left\{ \frac{C_1}{c_1}, \frac{C_2}{c_2}  \right\} ( c_1 \| \Delta x_{k^*+1}  \|^2 + c_2 \| \Delta x_{k^*} \|^2 )  +  \frac{ 4 (M_c^2 + \sigma^2) \epsilon^2 }{\rho^2 \sigma^4 } \\
& \overset{ \eqref{Parameter Ass.-inexact}, \eqref{ineq-1: complexity-inexact} }{\leq} \frac{ \max\{ C_1, C_2 \}/( \epsilon^{\eta}/16) }{ (16 \max\{ C_1, C_2 \}/ \epsilon^{\eta} ) )^2 } \cdot \frac{ \hat \alpha - \bar L }{K} +  \frac{ 4 (M_c^2 + \sigma^2) \epsilon^2 }{\rho^2 \sigma^4 } \\
 & \leq \frac{ (\hat \alpha - \bar L) \epsilon^\eta}{ 16 \max\{ C_1, C_2 \} K } + \frac{ 4 (M_c^2 + \sigma^2) }{\rho^2 \sigma^4 } \cdot \epsilon^2 \\
 & \overset{\eqref{eq:C1o}}{\le} \frac{ (\hat \alpha - \bar L) \epsilon^\eta }{ 16 C_1^o K } + \frac{ 4 (M_c^2 + \sigma^2) }{\rho^2 \sigma^4 } \cdot \epsilon^2 \\
 & \overset{ \eqref{Parameter Ass.-inexact} }{\leq} \frac{  (\hat \alpha - \bar L)  \epsilon^\eta }{ 16 C_1^o \hat \Delta \epsilon^{\eta - 2} } + \frac{\epsilon^2}{ 4 \rho R } \\
 & \overset{ (\rho \ge 1, R \ge 1) }{\leq} \frac{  (\hat \alpha - \bar L)  \epsilon^\eta }{ 16 C_1^o \hat \Delta \epsilon^{\eta - 2} } + \frac{\epsilon^2}{4} \\
 & \overset{\eqref{def: hatDta}}{\le} \frac{\epsilon^2}{2} + \frac{\epsilon^2}{4} < \epsilon^2.
\end{align*}
%% where the fifth to last inequality holds because $\max\{C_1, C_2 \} \ge C_1 \ge C_1^o$.
Therefore,  we have
\begin{align} \label{ineq: nn6}
T_\epsilon \leq k^*+1 \leq K + 1 = \lceil \hat \Delta/\epsilon^{ 2 - \eta} \rceil + 1.
\end{align}
It follows that 
$\hat{T}_\epsilon \le T_\epsilon \le \lceil \hat \Delta /\epsilon^{ 2 - \eta} \rceil + 1$,
completing the proof. \qed
\end{proof}

\begin{remark} (i).
  \rev{The condition $ \| c(x_0) \|^2 \le \min\{ C_0 / \rho, 1 \}$ is
    obviously satisfied by a feasible point, for which $c(x_0) =
    0$. In this case, we do not need Assumption~\ref{Ass: f.ub} and
    can prove a result with $ \hat{\alpha} = 7f(x_0) + 7C_0 - 6 \bar L
    + 13 \| \lambda_0 \|^2 + 2 $. An initial phase can be applied, if
    necessary, to find a point with small $\| c(x_0) \|$; we discuss
    this point in a later section.} \\
  (ii). When $\eta = 0$, the
  complexity result is consistent with that of
  \cite{pmlr-v70-hong17a}. \Rev{However, our parameter choices $\beta
    = \epsilon^\eta$ for $\eta>0$ allows us to choose $\beta$ to be
    small because, unlike \cite{pmlr-v70-hong17a}, we are not
    concerned with maintaining  strong convexity of the subproblem in
    Step 1 of Algorithm~\ref{Alg: Prox-PDA}.} Another benefit of small
  $\beta$ is that it allows complexity results to be proved for 
  $\epsilon$-2o points, which follows from part (ii) of Theorem~\ref{thm: first-order complexity - inexactsubproblem}, as we see next.
\end{remark}

\paragraph{Second-order complexity.}
We further assume that in Step 1 of Algorithm~\ref{Alg: Prox-PDA},
$x_{k+1}$ satisfies the following approximate second-order optimality
conditions:
\begin{align}\label{eq:2oi}
\nabla_{xx}^2 \sL_{\rho}(x_{k+1}, \lambda_k ) + \beta I \succeq -\epsilon_{k+1}^H I, \quad \mbox{for all $k \ge 0$,}
\end{align}
where $\{ \epsilon_{k+1}^H \}_{k \geq 0}$ is a chosen error sequence.

In corresponding fashion to the definition of $\hat{T}_\epsilon$ in
\eqref{def: hTeps}, we define $\widetilde{T}_{\epsilon}$  as follows:
\begin{align} \label{def: tTeps}
\widetilde{T}_{\epsilon} \triangleq \inf\{ t \geq 1 \mid x_t \mbox{ is an $\epsilon$-2o solution of \eqref{eqcons-opt}} \}.
\end{align}

We have the following result for complexity of obtaining an
$\epsilon$-2o stationary point of \eqref{eqcons-opt} through
Algorithm~\ref{Alg: Prox-PDA}.
%v
%
\begin{corollary} [Second-order complexity] \label{Corr: 2complexity - inexact}
Suppose that all assumptions and setting in Theorem~\ref{thm:
  first-order complexity - inexactsubproblem} hold. Assume that, in
addition, Step 1 of Algorithm~\ref{Alg: Prox-PDA} satisfies
\eqref{eq:2oi}, with $\epsilon_k^H \equiv \epsilon/2$ for all $k \ge
1$. Let $\eta \in [1,2]$. Then for $\widetilde{T}_{\epsilon}$ defined
in \eqref{def: tTeps}, we have $\widetilde{T}_{\epsilon} =
\sO(1/\epsilon^{2-\eta})$.
\end{corollary}
\begin{proof}
Since $\beta = \epsilon^\eta/2 \leq \epsilon/2$ and $ \epsilon_{k+1}^H
\equiv \epsilon/2$, for any $k \ge 0$, we have from \eqref{eq:2oi} that
\begin{align*}
\nabla_{xx}^2 \Lscr_{\rho}(x_{k+1}, \lambda_k)  \succeq - (\beta + \epsilon_{k+1}^H)I \succeq - \epsilon I.
\end{align*}
This fact indicates that 
\begin{align*}
\nabla^2 f(x_{k+1}) + \sum_{i = 1}^m [\lambda_{k+1}]_i \nabla^2 c_i(x_{k+1}) + \rho \nabla c(x_{k+1}) [ \nabla c(x_{k+1}) ]^T \succeq -\epsilon I,
\end{align*}
which implies that 
\begin{align*}
d^T ( \nabla^2 f(x_{k+1}) + \sum_{i = 1}^m [\lambda_{k+1}]_i \nabla^2 c_i(x_{k+1}) ) d & \geq -\epsilon \| d \|^2,
\end{align*}
for any $d \in S(x_{k+1}) \triangleq \{ d \in \bR^n \mid [\nabla
  c(x_{k+1})]^T d = 0 \}$. This is exactly condition
\eqref{ineq3:epsilonKKT2} of Definition~\ref{Def:
  epsilonKKT2}. Therefore, we have
\begin{align*}
\widetilde{T}_{\epsilon} & = \inf \{ t \geq 1 \mid \exists \lambda \in \bR^{ m }, \| \nabla f(x_t ) + \nabla c(x_t ) \lambda \| \leq \epsilon, \| c(x_t) \| \leq \epsilon, \\
& \hspace{+1in} d^T ( \nabla^2 f(x_t) + \sum_{i = 1}^m \lambda_i \nabla^2 c_i( x_t ) ) d \geq -\epsilon \| d \|^2, \;\; \mbox{for all $d \in S(x_t)$} \} \\
& \leq \inf \{ t \geq 1 \mid \| \nabla f(x_t ) + \nabla c(x_t ) \lambda_t \| \leq \epsilon, \| c(x_t) \| \leq \epsilon, \\
& \hspace{+1in} d^T ( \nabla^2 f(x_t) + \sum_{i = 1}^m [\lambda_t]_i \nabla^2 c_i( x_t ) ) d \geq -\epsilon \| d \|^2, \;\; \mbox{for all $d \in S(x_t)$} \} \\
& = \inf \{ t \geq 1 \mid \| \nabla f(x_t) + \nabla c(x_t) \lambda_t \| \leq \epsilon, \| c(x_t) \| \leq \epsilon \} = T_\epsilon.
\end{align*}
The result now follows from Theorem~\ref{thm: first-order complexity -
  inexactsubproblem}. \qed
\end{proof}

\section{Total iteration/operation complexity of Proximal AL} \label{sec: total.iter.comp.}

%In this section, we describe
%total iteration and operation complexity, when the subproblems are
%solved with the Newton-CG algorithm of \cite{Royer2019}.

In this section, we will choose an appropriate method to solve the
subproblem and estimate the total iteration and operation complexity
of our Proximal AL approach to find an $\epsilon$-1o or $\epsilon$-2o
solution. \rev{To solve the subproblem at each major iteration of
  Algorithm~\ref{Alg: Prox-PDA}, we can use methods for unconstrained
  smooth nonconvex optimization that allow the decrease condition
  \eqref{eq:decr} to hold, and approximate optimality conditions
  \eqref{eq:1oi} or \eqref{eq:2oi} to be enforced in a natural way,
  finding points that satisfy such conditions within a certain number
  of iterations that depends on the tolerances
  (\cite{cartis2012complexity,Ghadimi2016,Royer2019,GraN17a,birgin2017worst}).} Among
  these, the Newton-CG method described in \cite{Royer2019} has good
complexity guarantees as well as good practical performance.

To review the properties of the algorithm in \cite{Royer2019}, we
consider the following unconstrained problem:
\begin{align}\label{opt: uncon}
\min_{z \in \bR^n} & \quad F(z)
\end{align}
where $F: \bR^n \rightarrow \bR$ is a twice Lipschitz continuously
differentiable function. The Newton-CG approach makes use of the
  following assumption.
\begin{assumption}\label{Ass: NCG}
  \noindent
  \begin{itemize}
\item[(a)] The set $\{ z \mid F(z) \le F(z_0) \}$ is compact, where
  $z_0$ is the initial point.
\item[(b)] $F$ is twice uniformly Lipschitz continuously
  differentiable on a neighborhood of $\{ z \mid F(z) \le F(z_0) \}$
  that includes the trial points generated by the algorithm.
\item[(c)] Given $\epsilon_H > 0$ and $0 < \delta \ll 1$, a procedure
  called by the algorithm to verify approximate positive definiteness
  of $\nabla^2 F(z)$ either certifies that $\nabla^2 F(z) \succeq
  -\epsilon_H I$ or finds a direction along which curvature of
  $\nabla^2 F(z)$ is smaller than $-\epsilon_H/2$ in at most
$\Nmeo \triangleq \min \{ n, 1 + \lceil \Cmeo \epsilon_H^{-1/2} \rceil \}$
Hessian-vector products, with probability $1-\delta$, where $\Cmeo$
depends at most logarithmically on $\delta$ and $\epsilon_H$.
\end{itemize}
\end{assumption}
Based on the above assumption, the following iteration complexity is
indicated by \cite[Theorem~4]{Royer2019}.
\begin{theorem} \label{thm: NCG} Suppose that Assumption~\ref{Ass:
    NCG} holds. The Newton-CG terminates at a point satisfying
\begin{align} \label{optcon: NCG}
\| \nabla F(z) \| \le \epsilon_g, \quad \lambda_{\min}(\nabla^2F(z))
\ge - \epsilon_H ,
\end{align}
in at most $\bar K$ iterations with probability at least
$(1-\delta)^{\bar K}$, where
\begin{align} \label{barK-bound}
\bar K \triangleq \left\lceil \CNCG \max\{ L_{F,H}^3, 1 \} ( F(z_0) - \Flow )\max\{ \epsilon_g^{-3} \epsilon_H^3, \epsilon_H^{-3} \} \right\rceil + 2.
\end{align}
(With probability at most $1 - (1-\delta)^{\bar K}$, it terminates
incorrectly within $\bar K$ iterations at a point at which $\| \nabla
F(z) \| \le \epsilon_g$ but $\lambda_{\min}(\nabla^2F(z)) < -
\epsilon_H$.) Here, $\CNCG$ is a constant that depends on
user-defined algorithmic parameters, $L_{F,H}$ is the Lipschitz
constant for $\nabla^2 F$ on the neighborhood defined in
Assumption~\ref{Ass: NCG}(b), and $\Flow$ is the lower bound of
$F(z)$.
\end{theorem}

Since in the Newton-CG approach, Hessian-vector products are the
fundamental operations, \cite{Royer2019} also derives operation
complexity results, in which the operations are either evaluations of
$\nabla F(z)$ or evaluations of matrix-vector products involving
$\nabla^2 F(z)$ and an arbitrary vector (which can be computed
without actually evaluating the Hessian itself).
\begin{corollary} \label{Corr: NCG} Suppose that Assumption~\ref{Ass:
    NCG} holds. Let $\bar K$ be defined as in \eqref{barK-bound}. Then
  with probability at least $(1- \delta)^{\bar K}$, Newton-CG
  terminates at a point satisfying \eqref{optcon: NCG} after at most
\begin{align*}
( \max\{ 2 \min\{ n, J(U_{F,H}, \epsilon_H) \} + 2, \Nmeo  \} ) \bar K 
\end{align*}
Hessian-vector products, where $U_{F,H}$ is the upper bound for
$\nabla^2F(z)$ on  the neighborhood defined in
Assumption~\ref{Ass: NCG}(b) and $J(\cdot,\cdot)$ satisfies
\begin{align} \label{def: J}
J(U_{F,H}, \epsilon_H) \le \min \left\{ n, \left\lceil \left( \sqrt{\kappa} + \frac{1}{2} \right) \log \left( \frac{144(\sqrt{\kappa} + 1)^2 \kappa^6 }{ \zeta ^2 } \right)  \right\rceil \right\},
%= \min\{n, \tilde{\sO}(\epsilon_H^{-1/2}) \}, \kappa \le \frac{U_{F,H} + 2\epsilon_H}{\epsilon_H}
\end{align} 
where $\kappa \triangleq \frac{ U_{F,H} + 2 \epsilon_H }{\epsilon_H}$
and $\zeta$ is a user-defined algorithmic parameter.  (With
probability at most $1 - (1-\delta)^{\bar K}$, it terminates
incorrectly within such complexity at a point for which
$\| \nabla F(z) \| \le \epsilon_g$ but
$\lambda_{\min}(\nabla^2F(z)) < - \epsilon_H$.) 
\end{corollary}

To get total iteration and operation complexity we can aggregate the
cost of applying Newton-CG to each subproblem in Algorithm~\ref{Alg:
  Prox-PDA}.  We present a critical lemma before deriving the total
iteration complexity and operation complexity (Theorem~\ref{thm: total
  complexity} and Corollary~\ref{corr: operation complexity}).  For
these purposes, we denote the objective to be minimized at iteration
\Rev{$k$} of Algorithm~\ref{Alg: Prox-PDA} as follows:
\begin{align} \label{def: psik}
\psi_k(x) \triangleq \Lscr_\rho(x,\lambda_k) + \frac{\beta}{2}\| x -
x_k \|^2.
\end{align}
Additionally, we recall from Assumption~\ref{Ass: Compact sublevel
  set} that
$S_{\alpha}^0 \triangleq \{ f(x) + \frac{\rho_0}{2} \| c(x) \|^2 \leq
\alpha \}$ is either empty or compact for all $\alpha$.
\begin{lemma} \label{lm: compact level set} Suppose that assumptions
  and parameter settings in Theorem~\ref{thm: first-order
      complexity - inexactsubproblem} hold. In addition, suppose that $\rho \ge \tfrac12 \| \lambda_0 \|^2  + \rho_0 $. Then we
  have
\[
\{ x \mid \psi_k(x) \le \psi_k(x_k) \} \subseteq
S_{ \hat \alpha }^0,
\]
and
\begin{align}\label{ineq: psidecbound}
\psi_k(x_k) - \psi_k^{low} \le \hat{\alpha} - \bar L,
\end{align}
for all $k \ge 0$, where $\psi_k^{low} \triangleq \inf_{x \in \bR^n} \psi_k(x) $ and $ \hat \alpha $ is defined in \eqref{eq:alpha_l0-inexact}. Hence $\{ x \mid \psi_k(x)
\le \psi_k(x_k) \}$ is compact for all $k \ge 0$.
\end{lemma}
\begin{proof}
Because of $c_1>0$ and $c_2>0$ and \eqref{ineq: nn7}, we
have for any $k \ge 1$ that
\begin{align} \label{ineq: nn2}
P_{k+1} \le P_1 + \frac{16 (M_c^2 + \sigma^2) R  }{ \rho \sigma^4} \overset{ \eqref{Parameter Ass.-inexact} }{\le} P_1 + 1.
\end{align}
Thus for any $k \ge 1$, we have
\[
  \psi_k(x_k) = \Lscr_{\rho}(x_k, \lambda_k) \le P_k \le P_1 + 1,
\]
which, using \eqref{Bound on P1} and \eqref{eq:alpha_l0-inexact}, implies that
\begin{equation}
\label{ineq: nn4}
 \psi_k(x_k) \le 7 \bar U + 7C_0 + 13 \| \lambda_0 \|^2 - 6 \bar L + 1 = \hat{\alpha}-1.
\end{equation}
Note that \eqref{ineq: nn4} also holds when $k = 0$ since
\begin{align*}
\psi_0(x_0) & = \Lscr_\rho(x_0,\lambda_0) \overset{\eqref{ineq: AL0.bd.}}{\le} \bar U + \frac{1}{2\rho} \| \lambda_0 \|^2 + C_0 \\
& \overset{\eqref{ineq: nn0}}{\le} \bar U + \frac{1}{2\rho} \| \lambda_0 \|^2 + C_0 + 6( \bar U + C_0 - \bar L + (5/4) \| \lambda_0 \|^2 )\\
& \overset{ (\rho \ge 1) }{\le} 7 \bar U + 7 C_0 - 6 \bar L + 8 \| \lambda_0 \|^2 <  \hat \alpha - 1
\end{align*}
Further, for any $k \ge 0$, we have
\begin{align}
\notag
  \psi_k(x) & = \sL_\rho( x, \lambda_k ) + \frac{\beta}{2} \| x - x_k \|^2 \\
  \notag
  & = f(x) + \frac{\rho}{2} \| c(x) \|^2 + \lambda_k^T c(x) + \frac{\beta}{2} \| x - x_k \|^2 \\
\notag
& \ge f(x) + \frac{\rho}{2} \| c(x) \|^2 - \frac{ \| \lambda_k \|^2 }{ 2(\rho - \rho_0) } - \frac{ ( \rho - \rho_0 ) \| c(x) \|^2 }{2} \\ 
\notag
& \ge f(x) + \frac{\rho_0}{2} \| c(x) \|^2 - \frac{1}{2 (\rho - \rho_0)} \max\left\{ \| \lambda_0 \|^2, \frac{(M_f + \beta D_S + 1)^2}{\sigma^2} \right\}  \\
\notag
& \overset{ \eqref{Parameter Ass.-inexact} }{\ge} f(x) + \frac{\rho_0}{2} \| c(x) \|^2 - \max\left\{ \frac{ \| \lambda_0 \|^2 }{ 2 (\rho - \rho_0) }, 1 \right\},  \\
\label{ineq: nn5}
& \overset{ ( 2( \rho - \rho_0) \ge \| \lambda_0 \|^2 ) }{=} f(x) + \frac{\rho_0}{2} \| c(x) \|^2 - 1.
\end{align}
The second inequality holds because
$\| \lambda_k \| \le (M_f + \beta D_S + 1)^2/\sigma^2$,
$\forall k \ge 1$ from Lemma~\ref{lm: xk.bd.}. Then, for any
$k \ge 0$, we have by combining \eqref{ineq: nn4} and \eqref{ineq:
  nn5} that
\begin{equation}
  \label{ineq: psidiff}
 \psi_k(x_k) - \psi_k(x)
   \le \hat \alpha - \left( f(x) + \frac{\rho_0}{2} \| c(x) \|^2 \right). 
\end{equation}
Thus, for any $k \ge 0$, we have
\[
\psi_k(x) \le \psi_k(x_k) \implies \psi_k(x_k) - \psi_k(x) \ge 0 \overset{\eqref{ineq: psidiff}}{\implies}
f(x) + \frac{\rho_0}{2} \| c(x) \|^2 \le \hat \alpha. 
\]
Therefore $\{ x \mid \psi_k(x) \le \psi_k(x_k) \} \subseteq
S_{\hat \alpha}^0$ for all $k \ge 0$. For the claim \eqref{ineq: psidecbound}, note that
\begin{align*}
\psi_k(x_k) - \psi_k^{low} & =  \sup_{x \in \bR^n} ( \psi_k(x_k) - \psi_k(x) ) \\
& \overset{\eqref{ineq: psidiff}}{\le}  \sup_{x \in \bR^n} \left( \hat \alpha - \left( f(x) + \frac{\rho_0}{2} \| c(x) \|^2 \right) \right) \overset{\eqref{eq:defL}}{=} \hat \alpha - \bar L.
\end{align*}
 \qed
\end{proof}

By Lemma~\ref{lm: compact level set}, we know that if the Newton-CG
method of \cite{Royer2019} is used to minimize $\psi_k(x)$ at
iteration \Rev{$k$} of Algorithm~\ref{Alg: Prox-PDA},
Assumption~\ref{Ass: NCG}(a) is satisfied at the initial point
$x_k$. It also shows that the amount $\psi_k(x)$ can decrease at
iteration \Rev{$k$} is uniformly bounded for any $k \ge 0$. This is
important in estimating iteration complexity of Newton-CG to solve the
subproblem.

%\begin{lemma}\label{lm: psidecbound}
%Suppose that the assumptions and parameter settings in Lemma~\ref{lm: lambound} hold. Then for any $k \ge 0$,
%where $\psi_k^{low} \triangleq \inf_{x \in \bR^n} \psi_k(x) $, and $\bar \alpha$ is defined as in Lemma~\ref{lm: compact level set}.
%\end{lemma}
  
The following assumption is needed to prove complexity results about
the Newton-CG method.  Recall from definition \eqref{def: psik} that
\begin{align}
  \label{def:Hpsik}
  & \nabla^2 \psi_k(x)  \\
  \nonumber
  & = \nabla^2 f(x) + \sum_{i
  = 1}^m [\lambda_k]_i \nabla^2 c_i(x)  + \rho \sum_{i=1}^m c_i(x)
\nabla^2 c_i(x) + \rho \nabla c(x) \nabla c(x)^T + \beta I.
\end{align}
\begin{assumption} \label{Ass: boundNb}
  \begin{itemize}
\item[(a)] There exists a bounded open convex neighborhood $\sN_{\hat \alpha}$ of $S_{\hat
      \alpha}^0$, where $\hat{\alpha}$ is defined as in \eqref{eq:alpha_l0-inexact}, such that for any \Rev{$k \ge 0$}, the trial points of Newton-CG in
  iteration $k$ of Algorithm~\ref{Alg: Prox-PDA} lie in $\sN_{\hat \alpha}$. Suppose that on $\sN_{\hat \alpha}$, the functions $f(x)$ and $c_i(x)$, $i=1,2,\dotsc,m$
  are twice uniformly Lipschitz continuously differentiable.
%  \yx{($\nabla c(x)$ is bounded and $c(x)$ is Lipschitz continuous.) \bf{I found that we do not need these because Assumption 1 (ii) and (iii).}}
\item[(b)] Given $\epsilon_{k+1}^H > 0$ and $0 < \delta \ll 1$ at
  iteration \Rev{$k \ge 0$}, the procedure called by Newton-CG to
  verify sufficient positive definiteness of $\nabla^2 \psi_k$ either
  certifies that $\nabla^2 \psi_k(x) \succeq -\epsilon_{k+1}^H I$
%% \sjw{Are you sure about the $k-1$ subscript here? Yue: I believe
%%   so. A easy way to see this: at the 1st iteration, we minimize
%%   $\Lscr_\rho(x, \lambda_0) + \frac{\beta}{2} \| x - x_0 \|^2 $,
%%   which is defined as $\psi_0(x)$ according to \eqref{def:
%%   psik}. Meanwhile, the precision parameters for this subproblem
%%   are $\epsilon_1$ and $\epsilon_1^H$.}
  or else finds a vector of curvature smaller than $-\epsilon_{k+1}^H/2$
  in at most
\begin{align} \label{def: Nmeo}
\Nmeo \triangleq \min \{ n, 1 + \lceil \Cmeo (\epsilon_{k+1}^H)^{-1/2} \rceil \}
\end{align} 
Hessian-vector products, with probability $1-\delta$, where $\Cmeo$
depends at most logarithmically on $\delta$ and $\epsilon_{k+1}^H $.
\end{itemize}
\end{assumption}

Boundedness and convexity of $\sN_{\hat \alpha}$
%% \sjw{Why do we need convexity of this neighborhood? Is this
%%   assumption needed by Newton-CG? It doesn't seem to be needed by
%%   Assumption 3(b). \\ Yue: We may need it to show Lipschitz
%%   continuity of $\nabla^2 \psi_k$. In Assumption 4(b), Lipschitz
%%   continuity of $\nabla^2 F$ is assumed.}  and boundedness of
%%   $\lambda_k$ (Lemma~\ref{lm: xk.bd.}), \sjw{I added this last
%%   phrase, OK? \\ Yue: This may not be necessary since $\lambda_k$
%%   is fixed for each $k$.}
and Assumption~\ref{Ass: boundNb}(a) imply that $\nabla^2 \psi_k(x)$
is Lipschitz continuous on $\sN_{\hat \alpha}$. Thus,
Assumption~\ref{Ass: NCG}(b) holds for each subproblem. Further, if we
denote the Lipschitz constant for $\nabla^2 \psi_k$ by $L_{k,H}$, then
there exist $U_1$ and $U_2$ such that
\begin{equation} \label{eq:sh8}
  L_{k,H} \le U_1 \rho + U_2,
\end{equation}
where $U_1$ and $U_2$ depend only on $f$ and $c$, $\sN_{\hat
    \alpha}$, and the upper bound for $\| \lambda_k \|$ from
Lemma~\ref{lm: xk.bd.}. Moreover, if $c(x)$ is linear, then
$L_{k,H} = L_H$, where $L_H$ is the Lipschitz constant for $\nabla^2
f$.

The next theorem analyzes the total iteration complexity, given the
parameter settings in Theorem~\ref{thm: first-order complexity -
  inexactsubproblem} (with some additional requirements).
\begin{theorem} \label{thm: total complexity}
Consider Algorithm~\ref{Alg: Prox-PDA} with stopping criterion
\eqref{stopcrit}, and suppose that the subproblem in Step 1 is solved
with the Newton-CG procedure such that $x_{k+1}$ satisfies
\eqref{eq:1oi}, \eqref{eq:decr} and with high probability satisfies
\eqref{eq:2oi}. Suppose that Assumption~\ref{Ass: Compact
    sublevel set}, Assumption~\ref{Ass: weak} with $\sS = S_{\hat
    \alpha}^0$ (with $\hat \alpha$ defined in
  \eqref{eq:alpha_l0-inexact}), \Rev{Assumption~\ref{Ass: f.ub}} and Assumption \ref{Ass: boundNb}
hold. $\epsilon \in (0,1]$ and $\eta \in [1,2]$ are
  given. In addition, let $\| \tilde{r}_k \| \le \epsilon_k^g \triangleq \min\{ 1/k, \epsilon/2 \}$, for all $k \ge 1$ (so that
  $R = \sum_{k=1}^\infty 1/k^2 = \pi^2/6$). Let $\beta = \epsilon^\eta/2$ and \Rev{assume that $\rho \in [\rho_{\eta}, C_\rho \rho_\eta]$, where $C_{\rho}>1$ is constant and
\begin{align} \label{Parameter Ass.-inexact3}
\begin{aligned}
 \rho_{\eta} & := \max\left\{ \frac{ 16 \max\{C_1, C_2 \}}{\epsilon^\eta}, \frac{(M_f + \beta D_S + 1)^2}{2 \sigma^2} + \rho_0, \right. \\
& \quad\quad\quad \left.   \rev{ \frac{\| \lambda_0 \|^2}{2} + \rho_0,} \frac{16(M_c^2 + \sigma^2) R }{ \sigma^4 }, 3\rho_0, 1 \right\},
\end{aligned}
\end{align}}
where $C_1$ and $C_2$ are defined in \eqref{eq:C1C2} and $D_S$ is the
diameter of $S_{\hat \alpha}^0$ (see \eqref{def: DS}). \rev{Suppose
  that $x_0$ satisfies $ \| c(x_0) \|^2 \le \min\{ C_0 /\rho , 1 \} $,
  where $C_0$ is the constant appearing in
  \eqref{eq:alpha_l0-inexact}.}  Then we have the following.
\begin{itemize}
\item[(i)] If we set $\epsilon_k^H \equiv \sqrt{\epsilon}/2$, then the
  total number of iterations of Newton-CG before Algorithm~\ref{Alg:
    Prox-PDA} stops and outputs an $\epsilon$-1o solution is
  $\sO(\epsilon^{ -2\eta - 7/2 })$, which is optimized when $\eta = 1$. When
  $c(x)$ is linear, this total iteration complexity is $\sO(\epsilon^{
    \eta - 7/2 })$, which is optimized when $\eta = 2$.
\item[(ii)] If we let $\epsilon_k^H \equiv \epsilon/2$, then the total
  iteration number before Algorithm~\ref{Alg: Prox-PDA} stops and
  outputs an $\epsilon$-1o solution with probability $1$ and an
  $\epsilon$-2o solution with probability at least $(1 - \delta)^{\bar
    K_{T_\epsilon}}$ is $\sO(\epsilon^{-2\eta-5})$, and $\bar
  K_{T_\epsilon} = \sO(\epsilon^{-3\eta - 3})$, where $T_\epsilon$ is
  defined in \eqref{def: Teps} and $\bar K_{T_\epsilon}$ is the
  iteration complexity at iteration \Rev{$T_{\epsilon}-1$}, defined
  below in \eqref{per.iter.complex}. This bound is optimized when
  $\eta = 1$. When $c(x)$ is linear, this complexity is
  $\sO(\epsilon^{\eta-5})$, and $\bar K_{T_\epsilon} =
  \sO(\epsilon^{-3})$, so the optimal setting for $\eta$ is $\eta = 2$
  in this case.
\end{itemize}
\end{theorem}
\begin{proof}
  We first prove (i). Note that if we use $x_k$ as the initial point
  for Newton-CG at iteration \Rev{$k$}, then \eqref{eq:decr} will be
  automatically satisfied because Newton-CG decreases the objective
  $\psi_k$ at each iteration. Due to Lemma~\ref{lm: compact level set}
  and Assumption~\ref{Ass: boundNb}, we know that Assumption~\ref{Ass:
    NCG} is satisfied for each subproblem. Thus, at iteration
  \Rev{$k$}, according to Theorem~\ref{thm: NCG}, given positive
  tolerances $\epsilon_g = \epsilon_{k+1}^g$
%% \sjw{WHAT IS $\epsilon_{k+1}$???? I spent ages trying to figure out
%%   what it means. It's not mentioned in the statement of this result
%%   but appears often in the proof. My best guess is that is ia
%%   somehow related to $\epsilon_g$. Is this it? Are we just assuming
%%   that $\epsilon_k \equiv \epsilon_g$? Don't we need to say this in
%%   the statement of this result??? This also appears to be an issue
%%   in the original version of the paper. OK, I see that it is kind
%%   of defined in the next sentence. I'll try to rewrite everything
%%   to make this all less confusing. For a start, I'll introduce
%%   notation $\epsilon_k^g$. \\ Yue: Ok, thanks for improving the
%%   readability.}
  and $\epsilon_H = \epsilon_{k+1}^H$, Newton-CG will terminate at a
  point $x_{k+1}$ that satisfies \eqref{eq:1oi} such that $\| \tilde
  r_{k+1} \| \leq \epsilon_{k+1}^g$ with probability 1, and that
  satisfies \eqref{eq:2oi} with probability $(1-\delta)^{\bar
    K_{k+1}}$, within
\begin{align} \label{per.iter.complex}
  & \bar K_{k+1}  \\
  \notag
  & \triangleq \left\lceil \CNCG \max\{ L_{k,H}^3, 1 \} ( \psi_k(x_k) - \psi_k^{low} ) \max\{ (\epsilon_{k+1}^g)^{-3} (\epsilon_{k+1}^H)^3, (\epsilon_{k+1}^H)^{-3} \} \right\rceil + 2.
\end{align}
iterations, where $L_{k,H}$ is the Lipschitz constant for $\nabla^2
\psi_k(x)$.  By substituting \eqref{ineq: psidecbound} from Lemma~\ref{lm: compact level set} into
\eqref{per.iter.complex}, we obtain
\begin{align} \label{ineq: Kk}
  \bar K_{k+1} \le \left\lceil \CNCG \max\{ L_{k,H}^3, 1 \} ( \hat \alpha - \bar L ) \max\{ (\epsilon_{k+1}^g)^{-3}
(\epsilon_{k+1}^H)^3, (\epsilon_{k+1}^H)^{-3} \} \right\rceil + 2,
\end{align}
for any $k \ge 0$. From \eqref{eq:sh8} and the conditions on $\rho$,
we have
\begin{equation} \label{eq:LkH}
  L_{k,H} \le U_1 \rho + U_2 =
  \sO(\epsilon^{-\eta}),
\end{equation}
where $U_1$ and $U_2$ depend only on $f$ and $c$, $\sN_{\hat \alpha}$, and the upper bound for $\| \lambda_k \|$ from
Lemma~\ref{lm: xk.bd.}. When $c(x)$ is linear, we have $L_{k,H} \equiv
L_H$.

Since $\| \tr_k \| \le \epsilon_k^g = \min\{ 1/k, \epsilon/2 \}$ for all $k \ge
1$, the definition of $T_\epsilon$ in \eqref{def: Teps} and the result
of Theorem~\ref{thm: first-order complexity - inexactsubproblem} imply
that $T_\epsilon = \sO(1/\epsilon^{2 - \eta})$. Therefore, for any
$k \le T_\epsilon$ and $\eta \in [1,2]$, we have 
\begin{equation} \label{eq:pd4}
1/k \ge 1/T_\epsilon = \Omega (\epsilon^{2-\eta}) \implies 
\epsilon_k^g = \Omega(\epsilon)
\implies
(\epsilon_k^g) ^{-1} = \sO(\epsilon^{-1}).
\end{equation}
Thus, when $\epsilon_k^H \equiv \sqrt{\epsilon}/2$, the term
  involving $\epsilon_{k+1}^g$ and $\epsilon_{k+1}^H$ on the
  right-hand sides of \eqref{per.iter.complex} and \eqref{ineq: Kk}
  are $\sO(\epsilon^{-3/2})$. Therefore, we have from the bound for
$\bar{K}_k$, the estimate \eqref{eq:LkH}, and $T_\epsilon =
\sO(1/\epsilon^{2 - \eta})$ that the total iteration complexity to
obtain an $\epsilon$-1o solution is
\[
\sum_{k=1}^{ T_\epsilon } \bar K_k = \sum_{k=1}^{ T_\epsilon } \max \{ L_{k-1,H}^3,1 \} \sO(\epsilon^{-3/2}) = T_\epsilon \sO(\epsilon^{-3\eta}) \sO(\epsilon^{-3/2}) = \sO(\epsilon^{-2\eta - 7/2}).
\]
This bound is optimized when $\eta = 1$. When $c(x)$ is linear, we
have from $L_{k,H}= L_H = \sO(1)$ that the complexity is
\begin{align*}
\sum_{k=1}^{ T_\epsilon } \bar K_k = \sum_{k=1}^{ T_\epsilon } \max\{ L_H^3, 1 \} \sO(\epsilon^{-3/2}) = T_\epsilon \sO(\epsilon^{-3/2}) = \sO(\epsilon^{ \eta - 7/2}).
\end{align*}
This bound is optimized when $\eta = 2$.

We turn now to (ii). Since Algorithm~\ref{Alg: Prox-PDA} stops at
iteration \Rev{$T_\epsilon-1$}, Newton-CG will stop at the point
$x_{T_\epsilon}$ satisfying \eqref{eq:1oi} with probability $1$ and
\eqref{eq:2oi} with probability at least $(1 -
\delta)^{\bar{K}_{T_\epsilon}}$. Since $\epsilon_{T_\epsilon}^H =
\epsilon/2$, $\eta \in [1,2]$, and $\beta = \epsilon^\eta/2 \le
\epsilon/2$, the following conditions are satisfied with probability
at least $(1 - \delta)^{\bar{K}_{T_\epsilon}}$:
%% \begin{align*}
%% & \nabla_{xx}^2 \Lscr_{\rho}(x_{T_\epsilon}, \lambda_{T_\epsilon - 1})  \overset{\eqref{eq:2oi}}{\succeq} - (\beta + \epsilon_{T_\epsilon}^H)I \succeq - \epsilon I, \\
%% \implies & \nabla^2 f(x_{T_\epsilon}) + \sum_{i = 1}^m [\lambda_{T_\epsilon}]_i \nabla^2 c_i(x_{T_\epsilon}) + \rho \nabla c(x_{T_\epsilon}) \nabla c(x_{T_\epsilon})^T \succeq -\epsilon I, \\
%% \implies & d^T \left( \nabla^2 f(x_{T_\epsilon}) + \sum_{i = 1}^m [\lambda_{T_\epsilon}]_i \nabla^2 c_i(x_{T_\epsilon}) \right) d \geq -\epsilon \| d \|^2,\\
%% & \quad \mbox{for any $d \in S(x_{T_\epsilon}) \triangleq \{ d \in \bR^n \mid [\nabla c(x_{T_\epsilon})]^T d = 0 \}$.}
%% \end{align*}
\begin{align*}
\nabla_{xx}^2 \Lscr_{\rho}(x_{T_\epsilon}, \lambda_{T_\epsilon - 1})  \overset{\eqref{eq:2oi}}{\succeq} - (\beta + \epsilon_{T_\epsilon}^H)I & \succeq - \epsilon I, \\
\implies \; \nabla^2 f(x_{T_\epsilon}) + \sum_{i = 1}^m [\lambda_{T_\epsilon}]_i \nabla^2 c_i(x_{T_\epsilon}) + \rho \nabla c(x_{T_\epsilon}) \nabla c(x_{T_\epsilon})^T & \succeq -\epsilon I, \\
\implies \; d^T \left( \nabla^2 f(x_{T_\epsilon}) + \sum_{i = 1}^m [\lambda_{T_\epsilon}]_i \nabla^2 c_i(x_{T_\epsilon}) \right) d & \geq -\epsilon \| d \|^2,\\
 \mbox{for any $d \in S(x_{T_\epsilon}) \triangleq \{ d \in \bR^n \mid [\nabla c(x_{T_\epsilon})]^T d = 0 \}$},&
\end{align*}
which matches condition \eqref{ineq3:epsilonKKT2} of
Definition~\ref{Def: epsilonKKT2}.  Therefore, $x_{T_\epsilon}$ is an
$\epsilon$-1o solution with probability 1 and an $\epsilon$-2o
solution with probability at least $(1 - \delta)^{\bar
  K_{T_\epsilon}}$. Since we have $(\epsilon_k^g)^{-1} =
  \sO(\epsilon^{-1})$ for $k \le T_{\epsilon}$ as in \eqref{eq:pd4},
  and $\epsilon_k^H = \epsilon/2$, the term involving
  $\epsilon_{k+1}^g$ and $\epsilon_{k+1}^H$ on the right-hand side of
  \eqref{per.iter.complex} and \eqref{ineq: Kk} is
  $\sO(\epsilon^{-3})$.  Recalling that $T_{\epsilon} =
  \sO(1/\epsilon^{2-\eta})$, the total iteration complexity to obtain
$x_{T_\epsilon}$
\begin{align*}
\sum_{k=1}^{ T_\epsilon } \bar K_k \overset{\eqref{ineq: Kk}}{=}
\sum_{k=1}^{ T_\epsilon } \max\{ L_{k-1,H}^3, 1 \} \sO(\epsilon^{-3})
\overset{\eqref{eq:LkH}}{=} T_\epsilon \sO(\epsilon^{-3 \eta})
\sO(\epsilon^{-3}) = \sO(\epsilon^{ -2\eta - 5}).
\end{align*}
This bound is optimized when $\eta = 1$. Note that
$$
\bar{K}_{T_\epsilon} \overset{\eqref{ineq: Kk}}{=} \max\{ L_{T_\epsilon -1,H}^3, 1 \} \sO(\epsilon^{-3}) \overset{\eqref{eq:LkH}}{=} \sO(\epsilon^{-3\eta - 3}).
$$
When $c(x)$ is linear, $L_{k,H}= L_H = \sO(1)$ and the
complexity to get $x_{T_\epsilon}$ is
\begin{align*}
\sum_{k=1}^{ T_\epsilon } \bar K_k \overset{\eqref{ineq: Kk}}{=} \sum_{k=1}^{ T_\epsilon } \max\{
L_H^3, 1 \} \sO(\epsilon^{-3}) = T_\epsilon \sO(\epsilon^{-3}) =
\sO(\epsilon^{ \eta - 5}),
\end{align*}
which is optimized when $\eta = 2$. Note that in this case 
$$
\bar K_{T_\epsilon} \overset{\eqref{ineq: Kk}}{=} \max\{ L_{H}^3, 1 \} \sO(\epsilon^{-3}) =
\sO(\epsilon^{- 3}).
$$ \qed
\end{proof}

\begin{remark}
  \rev{(i). A feasible point, if available, will satisfy $\|c(x_0)
    \|^2 \le \min\{ C_0 / \rho, 1 \}$ for all $\rho$. Otherwise, a
    ``Phase I'' procedure may be applied to the problem of minimizing
    $\| c(x) \|^2$. Since $\rho = \sO(\epsilon^{-\eta}) $, we have
    that $ C_0/\rho = \Omega( \epsilon^{\eta} )$, $\eta \in
    [1,2]$. Thus the Newton-CG algorithm could be use to find an
    approximate first-order point $\bar x$ such that $ \| \nabla_x (
    \| c(x) \|^2 ) \mid_{x = \bar x} \| = \| 2 \nabla c(\bar x) c(\bar
    x) \| \le \min \{\epsilon, \sqrt{C_0/\rho}, 1 \} =
    \Omega(\epsilon)$. If $\| c(\bar x) \| \le \min\{ \sqrt{ C_0/\rho
    }, 1 \}$, we can set $x_0 = \bar x$. Otherwise, $\| c(\bar x) \| >
    \min\{ \sqrt{ C_0/\rho }, 1 \} = \Omega(\epsilon)$, and we can
    terminate at the approximate infeasible critical point of $\| c(x)
    \|^2$, as in \cite{Car14a}. When Assumption~\ref{Ass: NCG} holds
    for $F(z) = \| c(z) \|^2$, Theorem~\ref{thm: NCG} indicates that
    the iteration complexity of Newton-CG to find $\bar{x}$ is
    $\sO(\epsilon^{-3/2})$ (where $\epsilon_g = \min \{\epsilon,
    \sqrt{ C_0/\rho }, 1 \}$, $\epsilon_H = \sqrt{\epsilon_g}$). Thus,
    the total iteration complexity of Proximal AL is not affected when
    we account for Phase 1.}\\ (ii). Note that when $c(x)$ is linear,
  the optimized total iteration complexity bounds to obtain
  $\epsilon$-1o and $\epsilon$-2o point are $\sO(\epsilon^{-3/2})$ and
  $\sO(\epsilon^{-3})$, respectively. These bounds match the best
  known ones in literature for linear constraints (see Table~\ref{tab:
    complex.}  and corresponding discussion in Section~\ref{subsec:
    relate.work.}). \Rev{When $c(x)$ is nonlinear, the optimized total
    iteration complexity bounds to locate $\epsilon$-1o and
    $\epsilon$-2o point are $\sO(\epsilon^{-11/2})$ and
    $\sO(\epsilon^{-7})$, respectively. These bound are not
    competitive with the evaluation complexity bounds derived for
    two-phase second-order methods in \cite{doi:10.1137/120869687} and
    \cite{Car19a} (see Table~\ref{tab: complex.}). These methods require solving a cubic regularization subproblem or minimizing a nonconvex program
    to global optimality per evaluation, which are potentially
    expensive computational tasks. The Newton-CG algorithm used for
    our inner loop has standard iterative linear algebra subproblems,
    and is equipped with worst-case operation complexity
    guarantees. We take up the issue of total operation complexity
    next.}
\end{remark}

Recalling the formula for $\nabla^2 \psi_k$ in \eqref{def:Hpsik}, we
define a constant $U_H$ such that
\begin{align}  \label{ineq: UH}
\| \nabla^2 \psi_k(x) \| \le U_H, \ \forall k \ge 0, \ \forall x \in \sS_{\hat \alpha}.
\end{align}
Since $f(x)$ and $c_i(x)$, $i=1,2,\dotsc,m$ are twice continuously
differentiable on a neighborhood $\sN_{\hat \alpha} \supseteq
  \sS_{\hat \alpha}$ (by Assumption~\ref{Ass: boundNb}), and
  $\sS_{\hat \alpha}$ is compact and $\lambda_k$ is upper bounded (by Lemma~\ref{lm: xk.bd.}), then such a $U_H > 0$
exists. Moreover, there exist quantities $\tilde{U}_1$, $\tilde{U}_2$
such that
\begin{equation} \label{eq:uc7}
  U_H \le \tilde{U}_1 \rho + \tilde{U}_2,
\end{equation}
where $\tilde{U}_1, \tilde{U}_2$ depend only $f(\cdot)$,
  $c(\cdot)$, $\sS_{\hat \alpha}$, $\beta$ (which is bounded if
equals to $\epsilon^\eta/2$ for all $\epsilon \le 1$ and $\eta \ge 0$), and
the upper bound for $\| \lambda_k \|$ in Lemma~\ref{lm: xk.bd.}.

We conclude this section with the result concerning operation
complexity of Algorithm~\ref{Alg: Prox-PDA} in which the subproblems
are solved inexactly with Newton-CG.
\begin{corollary} \label{corr: operation complexity}
%% Consider Algorithm~\ref{Alg: Prox-PDA} \sjw{You do mean Algorithm 2
%%   here right? I had to change it here and in the previous result.}
%% with Newton-CG applied to the subproblem in Step 1 such that $x_{k+1}$
%% satisfies \eqref{eq:decr}, \eqref{eq:1oi}, and (with high probability)
%% \eqref{eq:2oi}. Suppose that Assumption~\ref{Ass: uniform
%%   nondegeneracy}, Assumption~\ref{Ass: Compact sublevel set} and
%% Assumption~\ref{Ass: boundNb} hold and that $\epsilon \in (0,1)$ and
%% $\eta \in [0,2]$ are given. Suppose $\| \tilde{r}_k \| \le \min\{ 1/k,
%% \epsilon/2 \}$, for all $k \ge 1$. Suppose that $c(x_0) = 0$ and let
%% \[
%% \beta = \epsilon^\eta/2, \ \rho = \max\left\{ (32/\epsilon^\eta)\max\{C_1, C_2 \}, \frac{ \sqrt{8(M_c^2 + \sigma^2)} }{ \sigma^2 }, 3\rho_0, 1 \right\},
%% \]
%% where $C_1$ and $C_2$ are defined in \eqref{eq:C1C2}.
Suppose that the setup and assumptions of Theorem~\ref{thm: total
  complexity} are satisfied. $U_H$ is a constant satisfying
\eqref{ineq: UH} and \eqref{eq:uc7}. $J(\cdot, \cdot)$ and $\Nmeo$ are
specified in Corollary~\ref{Corr: NCG} and Assumption~\ref{Ass:
  boundNb}(b), respectively. Let $\Ktotal \triangleq
\sum_{k=1}^{T_{\epsilon}} \bar{K}_k$ denote the total iteration
complexity for Algorithm~\ref{Alg: Prox-PDA} with Newton-CG applied to
the subproblems, where $\bar{K}_k$ is defined as in
\eqref{per.iter.complex}. Then the following claims are true.
\begin{itemize}
\item[(i)] When $\epsilon_k^H \equiv \sqrt{\epsilon}/2$, then the
  total number of Hessian-vector products before Algorithm~\ref{Alg:
    Prox-PDA} stops and outputs an $\epsilon$-1o solution is bounded
  by
\begin{align*}
\max\{ 2 \min\{ n, J(U_H, \sqrt{\epsilon}/2) \} + 2, \Nmeo  \} \Ktotal.
\end{align*} 
%% This bound is $\sO(n \epsilon^{-2\eta -  7/2} )$ and
For all $n$ sufficiently large, this bound is $\tilde{\sO}(\epsilon^{
  -5\eta/2 - 15/4 })$ \\
(which reduces to $\tilde{\sO}(\epsilon^{ \eta/2
  - 15/4 })$ when $c(x)$ is linear).
%then for all $n$ sufficiently large, this bound is
%% $\sO( n \epsilon^{   \eta - 7/2 })$ and when $n$ is sufficiently large, equals to $\tilde{\sO}(\epsilon^{ \eta/2 - 15/4 })$.  
\item[(ii)] If we let $\epsilon_k^H \equiv \epsilon/2$, then the total number of Hessian-vector products before
  Algorithm~\ref{Alg: Prox-PDA} stops and outputs an $\epsilon$-1o
  solution with probability $1$ and $\epsilon$-2o with probability at least $(1 -
  \delta)^{\bar K_{T_\epsilon}}$ is bounded by
\begin{align*}
\max\{ 2 \min\{ n, J(U_H, \epsilon/2) \} + 2, \Nmeo  \} \Ktotal.
\end{align*}
%% This bound is $\sO( n \epsilon^{-2 \eta - 5})$
For all $n$ sufficiently large, this bound is
$\tilde{\sO}(\epsilon^{-5\eta/2 - 11/2})$\\
(which reduces to
$\tilde{\sO}(\epsilon^{\eta/2 - 11/2})$ when $c(x)$ is linear).
%% , where $\bar K_{T_\epsilon} = \sO(\epsilon^{-3\eta - 3})$.
%When $c(x)$ is linear, then for all $n$ sufficiently large, this bound is
%% $\sO( n \epsilon^{\eta - 5})$, and when $n$ is sufficiently large, equals to
%$\tilde{\sO}(\epsilon^{\eta/2 - 11/2})$.
%% , where $\bar K_{T_\epsilon} = \sO(\epsilon^{-3})$.
  \end{itemize}
\end{corollary}
\begin{proof}
%% \sjw{I'm going to try to simplify the proof so that it pertains
%%   just to ``sufficiently large $n$'' so that it is more
%%   comprehensible. I have commented out the previous proof. Yue:
%%   Thanks for your change. I think it's more readable now. I notice
%%   that you commented off the discussion of $\bar K_{T_\epsilon}$
%%   related to the probability to obtain $\epsilon$-2o, which is fine
%%   since we address this in the last Theorem. About the subscripts
%%   $k$, I disagree a little. Note that at iteration $k =
%%   1,2,\hdots$, we apply Newton-CG to $\min \psi_{k-1}(x)$ to obtain
%%   $x_k$ as the solution of precision $\| \tr_k \|$(first-o<rder)
%%   and $\epsilon_k^H$(second-order). And Newton CG stops within
%%   $\bar{K}_k$ iterations.}
Since $\{ \psi_k(x) \le \psi_k(x_k) \} \subseteq S_{\hat \alpha}^0 $ (Lemma~\ref{lm: compact level set}), then $\| \nabla^2 \psi_k(x) \| \le U_H$
on $\{ \psi_k(x) \le \psi_k(x_k) \}$ for each $k \ge 0$. Therefore,
from Corollary~\ref{Corr: NCG}, to solve the subproblem in iteration
\Rev{$k-1$} of Algorithm~\ref{Alg: Prox-PDA} \Rev{(for $k \ge 1$)}, Newton-CG requires at most
  \begin{equation} \label{eq:jd3}
( \max\{ 2 \min\{ n, J(U_H, \epsilon_{k}^H) \} + 2, \Nmeo  \} ) \bar K_k
\end{equation}
Hessian-vector products, where $\bar{K}_k$ is defined in \eqref{per.iter.complex}, and $J(\cdot,\cdot)$ is bounded
as in \eqref{def: J}. From the latter definition and the fact that
$U_H = \sO(\rho) = \sO(\epsilon^{-\eta})$, we have for sufficiently
large $n$ that
\begin{equation} \label{eq:jd4}
  J(U_H, \epsilon_{k}^H ) \le \min \left( n, \tilde{\sO}
  ((U_H/\epsilon_{k}^H)^{1/2}) \right) = \tilde{\sO} \left( ( \epsilon_{k}^H)^{-1/2}
    \epsilon^{-\eta/2} \right).
\end{equation}
From \eqref{def: Nmeo}, we have at iteration \Rev{$k-1$}, for sufficiently
large $n$, that
\begin{equation} \label{eq:jd5}
  \Nmeo = \min \left(n, \tilde{\sO}((\epsilon_{k}^H)^{-1/2}) \right) =
  \tilde{\sO} ((\epsilon_{k}^H)^{-1/2}).
\end{equation}
By noting that the bound in \eqref{eq:jd4} dominates that of \eqref{eq:jd5}, we have from \eqref{eq:jd3} that the number of
Hessian-vector products needed at iteration \Rev{$k-1$} is bounded by
\begin{equation} \label{eq:jd6}
   \tilde{\sO} \left( ( \epsilon_k^H )^{-1/2} \epsilon^{-\eta/2}
   \right) \bar K_k.
\end{equation}

To prove (i), we have $\epsilon_k^H = \sqrt{\epsilon}/2$, so by
substituting into \eqref{eq:jd6} and summing over
$k=1,2,\dotsc,T_{\epsilon}$, we obtain the following bound on the
total number of Hessian-vector products before termination:
\begin{equation} \label{eq:jd7}
  \tilde{\sO} (\epsilon^{-\eta/2-1/4}) \Ktotal,
\end{equation}
where $\Ktotal = \sO(\epsilon^{ - 2 \eta - 7/2 } )$ from
Theorem~\ref{thm: total complexity}(i). By substituting into
\eqref{eq:jd7}, we prove the result. When $c(x)$ is linear, we obtain the
tighter bound by using the estimate $\Ktotal = \sO(\epsilon^{ \eta -
  7/2 } )$ that pertains to this case.

For (ii), we have from Theorem~\ref{thm: total complexity}(ii) that
$x_{T_\epsilon}$ is an $\epsilon$-1o solution with probability 1 and
an $\epsilon$-2o solution with probability at least $(1 -
\delta)^{\bar K_{T_\epsilon}}$. By substituting $\epsilon_k^H =
\epsilon/2$ into \eqref{eq:jd6} and summing over $k=
1,\dotsc,T_{\epsilon}$, we have that the total number of
Hessian-vector products before termination is bounded by
\begin{equation} \label{eq:jd8}
  \tilde{\sO} (\epsilon^{-\eta/2-1/2}) \Ktotal,
\end{equation}
where $\Ktotal = \sO(\epsilon^{-2\eta - 5})$ from Theorem~\ref{thm:
  total complexity}(ii), so the result is obtained by substituting
into \eqref{eq:jd8}. When $c(x)$ is linear, we obtain the tighter
bound by using the estimate $\Ktotal = \sO(\epsilon^{\eta-5} )$ that
pertains to this case. \qed
\end{proof}

\section{Determining $\rho$} \label{sec: varyrho}
\Rev{
Our results above on outer iteration complexity, total iteration
complexity, and operation complexity for Algorithm~\ref{Alg: Prox-PDA}
are derived under the assumption that $\rho$ is larger than a certain
threshold. However, this threshold cannot be determined a priori
without knowledge of many parameters related to the functions and the
algorithm. In this section, we sketch a framework for determining a
sufficiently large value of $\rho$ without knowledge of these
parameters. This framework executes Algorithm~\ref{Alg: Prox-PDA} as
an inner loop and increases $\rho$ by a constant multiple in an outer
loop whenever convergence of Algorithm~\ref{Alg: Prox-PDA} has not
been attained in a number of iterations set for this outer loop.  The
framework is specified as Algorithm~\ref{Alg: Prox-PDA-2}.
\begin{algorithm}
\caption{ Proximal AL with trial value of $\rho$}\label{Alg: Prox-PDA-2}
\be
\item[0.] Choose initial multiplier $\Lambda_0$, positive sequences $\{ \rho_\tau \}_{\tau \ge 1}$ and $\{ T_\tau \}_{\tau \ge 1}$; set $\tau \leftarrow 1$.  
\item[1.] Call Algorithm~\ref{Alg: Prox-PDA} with $x_0 = z_\tau$, $\lambda_0 = \Lambda_0$, $\rho = \rho_\tau$ and run Algorithm~\ref{Alg: Prox-PDA} for $T_\tau$ number of iterations, or until the stopping criteria are satisfied.
  \item[2.] If the stopping criterion of Algorithm~\ref{Alg: Prox-PDA}
    are satisfied, STOP the entire algorithm and output solutions
    given by Algorithm~\ref{Alg: Prox-PDA}; otherwise,
    $\tau\leftarrow\tau+1$ and return to Step 1.  \ee
\end{algorithm}
The next theorem shows that $\{ \rho_{\tau} \}_{\tau \ge 1}$ and $\{
T_{\tau} \}_{\tau \ge 1}$ can be defined as geometrically increasing
sequences without any dependence on problem-related parameter, and
that this choice of sequences leads to an iteration complexity for
Algorithm~\ref{Alg: Prox-PDA-2} that matches that of
Algorithm~\ref{Alg: Prox-PDA} (from Theorem~\ref{thm: first-order
  complexity - inexactsubproblem}) to within a logarithm factor.
\begin{theorem} \label{thm: outer.iter.comp.recover}
%%   Consider Algorithm~\ref{Alg: Prox-PDA-2}.
Suppose that all the assumptions and settings in Theorem~\ref{thm:
  first-order complexity - inexactsubproblem} for Algorithm~\ref{Alg:
  Prox-PDA} hold except for the choice of $\rho$. In particular, the
values of $\epsilon \in (0,1)$, $\eta \in [0,2] $, $\beta$ and $R$ are
the same in each loop of Algorithm~\ref{Alg: Prox-PDA-2}, and $z_{\tau}$ satisfies $\| c(z_{\tau}) \|^2 \le \min \{
  C_0/\rho_{\tau} , 1 \}$, where $C_0$ is the constant appearing in
  \eqref{eq:alpha_l0-inexact}. Suppose that Algorithm~\ref{Alg:
  Prox-PDA-2} terminates when the conditions \eqref{stopcrit} are
satisfied. For user-defined parameters $q > 1$ and $T_0 \in \bZ_{++}$,
we define the sequences $\{ \rho_{\tau} \}_{\tau \ge 1}$ and $\{
T_{\tau} \}_{\tau \ge 1}$ as follows:
\begin{align*}
\begin{cases}
\rho_\tau = \max\{ q^\tau \epsilon^{2 - 2\eta}, 1 \}, \ T_\tau = \lceil T_0 q^\tau \rceil + 1, & \mbox{ if } \eta \in [0,1), \\
\rho_\tau = q^\tau, \ T_\tau = \lceil T_0q^\tau \rceil + 1, & \mbox{ if } \eta = 1, \\
\rho_\tau = q^\tau, \ T_\tau = \max\{ \lceil T_0q^\tau \epsilon^{2\eta - 2} \rceil + 1, T_0 \}, & \mbox{ if } \eta \in (1,2].
\end{cases}
\end{align*}
Then Algorithm~\ref{Alg: Prox-PDA-2} stops within $ \log_q \big(
\epsilon^{ \min\{\eta-2, -\eta \} } \big) + \sO(1)$ iterations.  The
number of iterations of Algorithm~\ref{Alg: Prox-PDA} that are
performed before Algorithm~\ref{Alg: Prox-PDA-2} stops is
$\tilde{\sO}( \epsilon^{\eta-2} )$.
\end{theorem}
\begin{proof}
According to Theorem~\ref{thm: first-order complexity -
  inexactsubproblem}, at iteration $\tau$, the stopping criterion must
be satisfied within $T_\tau$ number of iterations if $\rho$ satisfies
\eqref{Parameter Ass.-inexact} and $T_\tau$ is greater than the upper
bound for $T_\epsilon$ estimated in Theorem~\ref{thm: first-order
  complexity - inexactsubproblem} (see \eqref{ineq: nn6}). Therefore,
Algorithm~\ref{Alg: Prox-PDA-2} is guaranteed to stop when
$\rho_{\tau}$ and $T_\tau$ are large enough.\\
\indent When $\eta \in (0,1)$, $\rho_\tau$ will satisfy \eqref{Parameter Ass.-inexact} if
\begin{align*}
q^\tau \epsilon^{2 - 2 \eta} & \ge \max\left\{ \frac{16 \max\{ C_1, C_2 \}}{ \epsilon^\eta }, \frac{(M_f + \beta D_S + 1)^2}{2 \sigma^2} + \rho_0, \frac{16(M_c^2 + \sigma^2) R }{ \sigma^4 }, 3 \rho_0, 1 \right\} \\
\Leftrightarrow \tau & \ge \max \left\{ \log_q \left( 16 \max\{ C_1, C_2 \} \epsilon^{\eta -2 } \right), \right. \\
& \qquad \left. \log_q \left( \max \left\{ \frac{(M_f + \beta D_S + 1)^2}{2 \sigma^2} + \rho_0, \frac{16(M_c^2 + \sigma^2) R }{ \sigma^4 }, 3 \rho_0, 1 \right\} \epsilon^{2\eta - 2} \right) \right\} \\
& = \log_q( \epsilon^{\eta - 2} ) + \sO(1).
\end{align*}
$T_\tau$ will be larger than the upper bound for $T_\epsilon$, that
is, $ \lceil \hat \Delta \epsilon^{\eta - 2} \rceil+ 1$ (for $\hat
\Delta$ defined in \eqref{def: hatDta}) if
\begin{align*}
T_0 q^\tau \ge \hat \Delta \epsilon^{\eta - 2} \Leftrightarrow \tau \ge \log_q \left( \hat \Delta \epsilon^{\eta - 2}/T_0 \right) = \log_q( \epsilon^{\eta - 2} ) + \sO(1).
\end{align*}
Therefore, Algorithm~\ref{Alg: Prox-PDA-2} will stop in $\log_q(
\epsilon^{\eta - 2} ) + \sO(1)$ number of iterations. Note that
$\rho_\tau = \sO( \epsilon^{-\eta} )$, $ T_\tau = \sO( \epsilon^{\eta
  - 2} ) $ for any $\tau$ before the algorithm stops. Therefore the
total number of iteration of Algorithm~\ref{Alg: Prox-PDA} is
$\tilde{\sO}(\epsilon^{\eta - 2})$. The same result holds when $\eta
\in [1,2]$ and the proof is similar (thus omitted).
\end{proof}
\remark{In Theorem~\ref{thm: outer.iter.comp.recover}, we almost
  recover the iteration complexity of Algorithm~\ref{Alg: Prox-PDA}
  derived in Theorem~\ref{thm: first-order complexity -
    inexactsubproblem}, except for a factor of $\log(1/\epsilon)$. The
  iteration complexity required to obtain $\epsilon$-2o
  (Corollary~\ref{Corr: 2complexity - inexact}) is immediate by
  Algorithm~\ref{Alg: Prox-PDA-2} if Step 1 of Algorithm~\ref{Alg:
    Prox-PDA} satisfies \eqref{eq:2oi}. To recover the iteration
  complexity of subproblem solver (Newton-CG) derived in
  Theorem~\ref{thm: total complexity}, we could use similar approach
  by setting a limit on the iteration of Newton-CG and increasing this
  limit geometrically with respect to $\tau$. The approach and
  analysis are quite similar to that presented above, so we omit the
  details.  }
}

\section{Conclusion} \label{sec: conclusion}

We have analyzed complexity of a Proximal AL algorithm to solve smooth
nonlinear optimization problems with nonlinear equality
constraints. Three types of complexity are discussed: outer iteration
complexity, total iteration complexity and operation complexity. In
particular, we showed that if the first-order (second-order)
stationary point is computed inexactly in each subproblem, then the
algorithm outputs an $\epsilon$-1o ($\epsilon$-2o) solution within
$\sO(1/\epsilon^{2-\eta})$ outer iterations ($\beta =
\sO(\epsilon^\eta)$, $\rho = \Omega (1/\epsilon^{\eta})$; $\eta \in [0,2]$
for first-order case and $\eta \in [1,2]$ for second-order case).  We
also investigate total iteration complexity and operation complexity
when the Newton-CG method of \cite{Royer2019} is used to solve the
subproblems. \Rev{A framework for determining the appropriate value of algorithmic parameter $\rho$ is presented, and we show that the
  iteration complexity increases by only a logarithmic factor for this
  approach by comparison with the version in which $\rho$ is known in
  advance.}
%Numerical experiments show good performance on dictionary learning.

There are several possible extensions of this work. First, we may
consider a framework in which $\rho$ is varied within
Algorithm~\ref{Alg: Prox-PDA}. Second, extensions to nonconvex
optimization problems with nonlinear {\em inequality} constraints
remain to be studied.

\section*{Acknowledgments}
Research supported by Award N660011824020 from the DARPA Lagrange
Program, NSF Awards
% IIS-1447449,
1628384, 1634597, and 1740707; and Subcontract 8F-30039 from Argonne
National Laboratory.

\bibliographystyle{spmpsci}
\bibliography{ref}

\appendix

\section*{Appendix: Proofs of Elementary Results} \label{sec: App}

%%\paragraph{Proof of Lemma~\ref{lm: bound for mul. diff.}.}
%%\paragraph{Proof of Lemma~\ref{lm: Pfunclowerbounded}.}
%% \paragraph{Proof of Theorem~\ref{thm: first-order complexity - constr.} (iii).}

\paragraph{Proof of Theorem~\ref{OptCon}}
\begin{proof}
Since $x^*$ is a local minimizer of \eqref{eqcons-opt}, it is the
unique global solution of
\begin{equation}\label{opt: local}
\min \, f(x) + \frac{1}{4} \| x - x^* \|^4  \quad
\st \;\; c(x) = 0, \;\; \| x - x^* \| \le \delta,
\end{equation}
for $\delta > 0$ sufficiently small. For the same $\delta$, we define
$x_k$ to be the global solution of
\begin{equation}\label{opt: klocal}
  \min \, \quad f(x) + \frac{\rho_k}{2} \| c(x) \|^2 + \frac{1}{4} \|
  x - x^* \|^4 \quad \st \;\; \| x - x^* \| \le \delta,
\end{equation}
for a given $\rho_k$, where $\{ \rho_k \}_{k \ge 1}$ is a positive
sequence such that $\rho_k \to +\infty$. Note that $x_k$ is well
defined because the feasible region is compact and the objective is
continuous. Suppose that $z$ is any accumulation point of $\{ x_k
\}_{k \ge 1}$, that is, $x_k \to z$ for $k \in \sK$, for some
subsequence $\sK$. Such a $z$ exists because $\{ x_k \}_{k \ge 1}$
lies in a compact set, and moreover, $\| z - x^* \| \le \delta$. We
want to show that $z = x^*$. By the definition of $x_k$, we have for
any $k \ge 1$ that
\begin{align}
\label{ineq: nn9}
f(x^*) & = f(x^*) + \frac{\rho_k}{2}\| c(x^*) \|^2 + \frac{1}{4} \| x^* - x^* \|^4 \\
\notag
& \ge f(x_k) + \frac{\rho_k}{2}\| c(x_k) \|^2 + \frac{1}{4} \| x_k - x^* \|^4 
 \ge f(x_k) + \frac{1}{4} \| x_k - x^* \|^4.
\end{align}
By taking the limit over $\sK$, we have $f(x^*) \ge f(z) + \frac{1}{4}
\| z - x^* \|^4$. From \eqref{ineq: nn9}, we have
\begin{align} \label{ineq: infeas}
\frac{\rho_k}{2}\| c(x_k) \|^2 \le f(x^*) - f(x_k) \le f(x^*) -
\inf_{k \ge 1} f(x_k) < +\infty.
\end{align}
By taking limits over $\sK$, we have that $c(z) = 0$. Therefore, $z$
is the global solution of \eqref{opt: local}, so that $z =
x^*$.

Without loss of generality, suppose that $x_k \to x^*$ and $\| x_k -
x^* \| < \delta$. By first and second-order optimality conditions for
\eqref{opt: klocal}, we have
\begin{align}
\label{eq: 1ok}
\nabla f(x_k) + \rho_k \nabla c(x_k) c(x_k) + \| x_k - x^* \|^2 (x_k - x^*) & = 0, \\
\notag
\nabla^2 f(x_k) + \rho_k \sum_{i=1}^m c_i(x_k) \nabla^2 c_i(x_k) + \rho_k \nabla c(x_k) [ \nabla c(x_k) ]^T & \\
\label{eq: 2ok}
+ 2(x_k-x^*)(x_k-x^*)^T + \| x_k-x^* \|^2 I & \succeq 0.
\end{align}
Define $\lambda_k \triangleq \rho_k c(x_k)$ and $\epsilon_k \triangleq
\max\{ \| x_k - x^* \|^3, 3 \| x_k - x^* \|^2, \sqrt{ 2( f(x^*) -
    \inf_{k \ge 1} f(x_k) )/\rho_k} \}$. Then by \eqref{ineq:
    infeas}, \eqref{eq: 1ok}, \eqref{eq: 2ok} and Definition~\ref{Def:
    epsilonKKT2}, $x_k$ is $\epsilon_k$-2o. Note that $x_k \to x^*$
  and $\rho_k \to +\infty$, so $\epsilon_k \to 0^+$.  \qed
\end{proof}

\paragraph{Proof of Lemma~\ref{lm: Lb}}
\begin{proof}
We prove by contradiction. Otherwise for any $\alpha$ we could select sequence $\{ x_k \}_{k \geq
  1} \subseteq S_{\alpha}^0$ such that $f(x_k) + \frac{\rho_0}{2} \|
c(x_k) \|^2 < - k$. Let $x^*$ be an accumulation point of $\{ x_k
\}_{k \geq 1}$ (which exists by compactness of $S_\alpha^0$). Then
there exists index $K$ such that $f(x^*) + \frac{\rho_0}{2} \| c(x^*)
\|^2 \geq -K + 1 > f(x_k) + \frac{\rho_0}{2} \| c(x_k) \|^2 + 1$ for
all $k \geq K$, which contradicts the continuity of $f(x) +
\frac{\rho_0}{2} \| c(x) \|^2$. \qed
\end{proof}

%\paragraph{Proof of Lemma~\ref{lm: bound for mul. diff.-inexactsubproblem}.}
%Similar to that of Lemma~\ref{lm: bound for mul. diff.} thus omitted.

%% \sjw{But there is the extra factor of $2$ in the $C_1$ and $C_2$
%%   terms. Shouldn't we at least explain this difference? Perhaps
%%   even give a sketch of the proof? Yue: Proof is given. Difference
%%   from the proof of Lemma~\ref{lm: bound for mul. diff.} has been
%%   highlighted.}

% Authors must disclose all relationships or interests that 
% could have direct or potential influence or impart bias on 
% the work: 
%
% \section*{Conflict of interest}
%
% The authors declare that they have no conflict of interest.

% BibTeX users please use one of
%\bibliographystyle{spbasic}      % basic style, author-year citations
%\bibliographystyle{spmpsci}      % mathematics and physical sciences
%\bibliographystyle{spphys}       % APS-like style for physics
%\bibliography{}   % name your BibTeX data base

% Non-BibTeX users please use
%\begin{thebibliography}{}
%%
%% and use \bibitem to create references. Consult the Instructions
%% for authors for reference list style.
%%
%\bibitem{RefJ}
%% Format for Journal Reference
%Author, Article title, Journal, Volume, page numbers (year)
%% Format for books
%\bibitem{RefB}
%Author, Book title, page numbers. Publisher, place (year)
%% etc
%\end{thebibliography}

\end{document}